\documentclass{amsart}
\usepackage[english]{babel}
\usepackage[latin1]{inputenc}
\usepackage{mathrsfs}
\usepackage{amsmath}
\usepackage{amsthm}
\usepackage{amssymb}
\usepackage{indentfirst}
\usepackage[dvips]{graphicx}
\usepackage[matrix,arrow,tips,curve]{xy}

\usepackage{ifthen}
\newcommand{\cit}[1]{{\rm \textbf{#1}}}

\newcommand{\Ref}[2]{\cit{%
\ifthenelse{\equal{#1}{thm}}{Theorem}{}%
\ifthenelse{\equal{#1}{prop}}{Proposition}{}%
\ifthenelse{\equal{#1}{lem}}{Lemma}{}%
\ifthenelse{\equal{#1}{cor}}{Corollary}{}%
\ifthenelse{\equal{#1}{defn}}{Definition}{}%
\ifthenelse{\equal{#1}{oss}}{Remark}{}%
\ifthenelse{\equal{#1}{sec}}{Section}{}%
\ifthenelse{\equal{#1}{ex}}{Example}{}%
\ifthenelse{\equal{#1}{conj}}{Conjecture}{}%
\ifthenelse{\equal{#1}{ssec}}{Subsection}{}%
\ifthenelse{\equal{#1}{tab}}{Table}{}%
\ifthenelse{\equal{#1}{cla}}{Claim}{}%
\  \ref{#1:#2}%
}}

\theoremstyle{plain} 
\newtheorem{prop}{Proposition}[section]
\newtheorem{thm}[prop]{Theorem}
\newtheorem{lem}[prop]{Lemma} 
\newtheorem{cor}[prop]{Corollary}

\newtheorem{rem}[prop]{Remark}

\theoremstyle{remark}

\theoremstyle{definition}

\newcommand{\hk}{hyperk\"{a}hler }

\newcommand{\kahl}{K\"{a}hler }

\newcommand{\kntipo}{$K3^{[n]}$ type}

\newcommand{\Kv}{{K_v}}
\newcommand{\Kvt}{\wt{K}_v}

\newcommand{\Z}{\mathbb{Z}}

\newcommand{\C}{\mathbb{C}}
\renewcommand{\P}{\mathbb{P}}

\newcommand{\mc}{\mathcal}

\newcommand{\Zt}{\widetilde{Z}}

\newcommand{ \ff} { \frac }

\newcommand{\be}{\begin{equation}}
\newcommand{\ee}{\end{equation}}
\renewcommand{\phi}{\varphi}

\newcommand{\wt}{\widetilde}

\DeclareMathOperator{\rk}{rk}
\DeclareMathOperator{\Hom}{Hom}

\DeclareMathOperator{\Pic}{Pic}

\DeclareMathOperator{\im}{Im}

\DeclareMathOperator{\Sym}{Sym}
\DeclareMathOperator{\alb}{alb}

\usepackage{fullpage}

\begin{document}

\title{The Hodge diamond of O'Grady's $6$--dimensional example}

\author{Giovanni Mongardi}
\address{Dipartimento di Matematica, Universit\`{a} degli studi di Milano,  Via Cesare Saldini 50, Milano 20133, Italia  }

\author{Antonio Rapagnetta}
\address{Dipartimento di Matematica, Universit\`a di Roma Tor Vergata, Via della ricerca scientifica 1, Roma  00133, Italia}

\author{Giulia Sacc\`{a}}
\address{Department of Mathematics, Stony Brook University, Stony Brook, NY 11974-3651}

\begin{abstract} We realize O'Grady's six dimensional example of irreducible
holomorphic symplectic manifold as a quotient of an IHS manifold of K3$^{[3]}$--type by a birational involution, thereby computing its Hodge numbers.

\end{abstract}

\maketitle

\section{Introduction}

In this paper we present a new way of obtaining O'Grady's six dimensional 
example of irreducible holomorphic symplectic manifold and use this to compute 
its Hodge numbers. Further applications, such as the description of the 
movable cone or the answer to Torelli--type questions for this deformation class of 
irreducible holomorphic symplectic manifolds, will be the topic of a subsequent paper.

Recall that an irreducible holomorphic symplectic manifold (IHS) is a simply connected compact 
K\"ahler manifold that has a unique up to scalar holomorphic symplectic form. 
They arise naturally as one of the three building blocks of manifolds with 
trivial first Chern class according to the Beauville--Bogomolov decomposition \cite{Bogo}, \cite{Beauville}, 
the other two blocks being Abelian varieties and Calabi-Yau manifolds. 
By definition, IHS manifolds are higher dimensional generalizations of K3 surfaces, 
moreover they  have a canonically defined quadratic form on their integral 
second cohomology group, which allows to speak of their periods and to develop their theory 
in a way which is analogous to the theory of $K3$ surfaces. 
The interested reader can see \cite{huy_basic} and \cite{og_k3} for a general introduction on the topic.

There are two deformation classes of IHS manifolds in every even dimension greater or equal to $4$, introduced by Beauville  in \cite{Beauville}. 
They are the Hilbert scheme of $n$ points on a K3 and the generalized Kummer variety of 
dimension $2n$ of an abelian surface (i.e. the Albanese fiber of the Hilbert scheme of $n+1$ 
points of the abelian surface). Elements of these two deformation classes have second Betti 
number equal to $23$ and $7$, respectively, and are referred to as IHS  manifolds 
of K3$^{[n]}$--type and of generalized Kummer type, respectively.
There are two more examples, found by O'Grady in  \cite{OG1} and \cite{OG2}, 
of dimension ten and six, respectively, which are obtained from  a symplectic resolution 
of some singular moduli spaces of sheaves on a K3 surface and on an 
abelian surface, respectively. They are referred to as the exceptional examples 
of IHS, and their deformation classes are denoted by OG10, respectively OG6.

These exceptional examples have not been studied as much and their geometries are less understood. Though their topological Euler 
characteristic is known,  see \cite{rap_phd} and \cite{mozo}, 
even other basic invariants such as their Hodge numbers have not been computed yet. 
In the case of manifolds of K3$^{[n]}$--type, the Hodge numbers were computed by G\"ottsche \cite{got_hod}. 
 
One of the main results of this paper is to realize O'Grady's six dimensional example 
as a quotient of an IHS manifold of K3$^{[3]}$--type by a birational symplectic involution: 
we therefore relate this deformation class to the most studied deformation class of IHS manifolds
and this allows us, by resolving the indeterminacy locus of the involution 
and by describing explicitly its fixed locus (which has codimension $2$), 
to compute the Hodge numbers. The involution we use was first introduced in \cite{rap_phd} to compute 
the Beauville-Bogomolov form for IHS of type OG6 and then used in \cite{mw} 
to determine a special subgroup of the automorphisms group of such manifolds. 

Recently, there has been
considerable interest in exhibiting and classifying symplectic automorphisms of 
IHS manifolds \cite{bs}, \cite{hm}, \cite{gregorio} and \cite{mon_auto}. 
Notice that quotients of IHS by symplectic automorphisms rarely admit a symplectic 
resolution since for this to happen the fixed locus has to be of codimension $2$ 
(see \cite{kawat} for one of the few cases where this happens). 
Our construction, however, indicates that ``quotients'' by \emph{birational} symplectic automorphisms 
can have a symplectic resolution, and thus they are potentially interesting. 
In upcoming work, we will study some of these birational morphisms 
for manifolds of K3$^{[n]}$--type.

Recall that O'Grady's six dimensional example is obtained as a symplectic resolution of a certain 
natural subvariety of a moduli space of sheaves on an abelian surface $A$. 
In order to describe how to obtain it as a ``quotient'' of another 
IHS by a birational symplectic automorphism, we first need to introduce some notation. 

Let $X$ be a K3 or an abelian surface. 
Fix an effective  Mukai vector\footnote{$v$ is effective if it is the Mukai vector of a coherent sheaf on $X$.} 
$v \in H^*_{alg}(X, \Z)$, with $v^2 \ge -2$, and let $H$ be a sufficiently general ample line bundle on $X$. 
It is well known \cite{mukai}, \cite{yoshioka}  that, if $v$ is primitive,  
the moduli space $M_{v}(X, H)$ of $H$--stable sheaves on $X$ with Mukai vector $v$ 
is a smooth projective manifold of dimension $v^2+2$ and that, if $v^2 \ge 0$, it admits a 
holomorphic symplectic form. If $X$ is a K3 surface, then $M_{v}(X, H)$ 
is an IHS variety of K3$^{[n]}$--type, for $n=v^2/2+1$. 
Whereas, if $X=A$  is an abelian surface and if $v^2 \ge 4$,  
there is a non trivial Albanese variety and, 
in order to get an \emph{irreducible} holomorphic symplectic manifold, one needs to consider a fiber
\be 
K_v(A, H):=\alb^{-1}( 0),
\ee
of the Albanese morphism (which is isotrivial)
\[
\alb : M_v(A, H) \to  A\times A^{\vee}.
\]
Recall that if  $v^2 \ge 6$, $K_v(A, H)$ is deformation equivalent 
to the generalized Kummer variety $K^{[n]}(A):=\sum^{-1}(0)$, 
where $\sum: A^{[n+1]} \to A$ is the summation morphism. 

If we consider an $H$--stable sheaf $F$ with a primitive Mukai vector $v_0$, 
then for $m \ge 2$, the sheaf $F^{\oplus m}$ is strictly $H$--semistable. 
Hence if we set $v=m v_0$, this sheaf determines a singular point of the moduli space $M_{v}(X, H)$, 
whose smooth locus still carries a holomorphic symplectic form. 
In \cite{OG1} and \cite{OG2}, O'Grady considered the case of $v_0=(1, 0, -1)$ and $m=2$, 
and showed that the singular symplectic variety $M_{v}(X, H)$ admits a symplectic resolution
$\wt{M}_{v}(X, H).$ 
For $X$ K3, this resolution gives a $10$--dimensional IHS manifold of type OG10. 
For $X=A$, fix $F_{0}\in M_{v}(A, H)$ and denote by $K_{v}(A, H)$ the fiber over $0$ of the isotrivial fibration
\begin{equation} \label{Kv}
\begin{array}{cccc}
\mathbf{a}_{v}: & M_{v}(A, H) & \longrightarrow & A\times A^{\vee} \\
& F & \longmapsto & (Alb(c_{2}(F)),\det(F)\otimes \det(F_{0})^{-1})
\end{array}
\end{equation}
where $Alb : CH_{0}(A)\rightarrow A$ is the Albanese homomorphism. The proper transform $\wt{K}_{v}(A, H)$ of $K_{v}(A, H)$ in $ \wt {M}_{v}(X, H)$
is smooth and the induced map  
\be \label{f}
f_{v}: \wt{K}_{v}(A, H) \to K_{v}(A, H),
\ee
 is a symplectic resolution. The gives the $6$--dimensional IHS  $\wt{K}_{v}(A, H)$, whose deformation type is called OG6,   that is the object of this paper. 

Lehn and Sorger proved in \cite{ls} that for \emph{any} primitive $v_0$, with $v_0^2=2$, 
the moduli space ${M_{2v_0}(X, H)}$ admits a symplectic resolution. 
Finally, Perego and the second named author  \cite{per_rap}  showed that for any choice of $v_0$, 
with $v_0^2=2$, on a $K3$ or abelian surface, 
the IHS manifolds that one gets are deformation equivalent to OG10 and OG6, respectively. 

When $A$ is a general principally polarized abelian surface and $M_{v}(X, H)$ parametrizes pure $1$--dimensional sheaves, the IHS manifold $\wt{K}_{v}(A, H)$
is the image of a degree $2$ rational map whose domain is an IHS manifold of K3$^{[3]}$--type as we now  briefly sketch. 

Let us  consider a principal polarization $\Theta \subset A$. 
The Mukai vector $v_0=(0, \Theta, 1)$ satisfies $v_0^2=2$, and hence, 
if we set $v=2v_0$, there is symplectic resolution $\wt K_v \to K_v$ that is deformation equivalent to OG6. 
There is a natural support morphism $K_v \to |2\Theta|=\mathbb P^3$, realizing $K_v$ as a Lagrangian fibration. By definition of $K_v$, 
the fiber over a smooth curve $C \in |2\Theta|$ is the kernel of the natural morphism 
$\Pic^{6}(C) \to A$ (which is also the restriction of  $\mathbf{a}_{v}$  to $\Pic^{6}(C)\subset M_{v}(A, H)$).

It is well known that the morphism associated to the linear system $|2\Theta|$ 
is the quotient morphism $A \to A \slash \pm 1 \subset \mathbb P^3$ 
onto the singular Kummer surface of $A$. Let $S \to A \slash \pm 1$ be 
the minimal resolution of $A$. It is  well known that $S$, the Kummer surface of $A$, 
is a K3 surface. Notice that $S$ come naturally equipped with the 
degree $4$ nef line bundle $D$ obtained by pulling back the 
hyperplane section of $A \slash \pm 1 \subset \mathbb P^3$. Consider the diagram
\[
q: A \to A/ \pm 1 \subset \P^3
\]
\be \label{diagramma SKA}
\xymatrix{
\wt A \ar[d]_b \ar[r]^a & S \ar[d]^p \\
A \ar[r]_{q\phantom{mmm}} & A /\pm 1
}
\ee
where $\wt A$ is the blow up of $A$ at its $16$ $2$--torsion points or, equivalently, 
the ramified cover of $S$ along the exceptional curves $E_1, \dots, E_{16}$ of $p$. 
Consider the moduli space $M_w(S)$ of sheaves on $S$ with Mukai vector $w=(0, D, 1)$ that are 
stable with respect to a choosen, sufficiently general, polarization. 
This is an IHS manifold  birational to the Hilbert cube of $S$ 
and it has a natural morphism $M_w(S) \to |D|=\mathbb P^3$ realizing it as the 
relative compactified Jacobian of the linear system $|D|$ (also a Lagrangian fibration).

The morphisms in diagram \ref{diagramma SKA} induce a rational generically $2:1$ map
\[
b_* a^*=q^* p_*: M_w(S) \dashrightarrow M_{v}(A,H).
\]
Since $M_w(S)$ is simply connected, the image of this map lies in a fiber of $\mathbf{a}_{v}$, 
giving a $2:1$ morphism $ \Phi: M_w(S) \dashrightarrow K_{v}(A,H)$
. 

On the smooth fibers, this maps restricts to the natural $2:1$  pull back 
morphism $\Pic^3(C') \to \Pic^6(C)$, whose image is precisely  $\ker[\Pic^6(C) \to A]$.  
Recall that $\sum_i E_i$ is divisible by $2$ in $H^2(S, \Z)$ and that the line bundle 
$\eta:=\mc O_S( \ff{1}{2} \sum E_i)$ determines the double cover $q$.  
It follows that  the involution on $M_w(S)$ corresponding to $\Phi$ is given by tensoring 
by $\eta$ and $\wt{K}_{v}(A,H)$ is a birational model of the ``quotient'' of $M_w(S)$
by the birational involution induced by tensorization by  $\eta$.



In this paper,
for any Abelian surface $A$ and for an effective Mukai vector $v=2v_0$ with $v_0^2=2$ on $A$, 
we show that $\wt{K}_{v}(A,H)$ admits a rational double cover from  an IHS manifold $\underline{Y}_{v}(A,H)$ of 
K3$^{[3]}$--type. 
Recall that the singular locus $\Sigma_v \subset K_{v}(A, H)$ has codimension $2$ and 
can be identified with $A \times A^\vee \slash \pm1$ (for more details, see Section \ref{sec:reso}).
Following  \cite{OG2}, the symplectic resolution (\ref{f}) can be  obtained by two subsequent blow ups followed by a contraction: 
first one blows up the singular locus of $\Sigma_v $, then one blows up the proper transform of $\Sigma_v$ itself (which is smooth); 
 these two operations produce a manifold $\widehat{K}_{v}(A,H)$ that has a holomorphic 
two form degenerating along the strict transform of the exceptional divisor of 
the first blow up; contracting this exceptional divisor finally gives the manifold 
$\widetilde{K}_{v}(A, H)$ that has non--degenerate (hence symplectic) two form and a 
regular morphism $\wt K \to K$ which is, therefore, a symplectic resolution. 
The inverse image $\widehat{\Sigma}$ of $\Sigma_v$ in  $K_{v}(A, H)$ is a smooth divisor,  
which is divisible by two in the integral cohomology by results of the second named author \cite{rap_phd}. We show that the associated ramified double cover is a smooth manifold
birational to an IHS manifold of K3$^{[3]}$--type, which we denote by $\underline{Y}_{v}(A,H)$ and which is equipped with a birational symplectic  involution. 

This enable us to reconstruct $\widetilde{K}(A, H)$ starting from $\underline{Y}_{v}(A,H)$,  and its symplectic  birational  involution 
\[\underline{\tau}_{v}:\underline{Y}_{v}(A,H)\rightarrow \underline{Y}_{v}(A,H).\]
More specifically, $\underline{Y}_{v}(A,H)$
contains 256 $\P^3$s,  the birational involution $\underline{\tau}_{v}$ is regular on the complement of these $\P^3$s, and, moreover, this involution lifts to a regular involution on the blow up $\overline{Y}_{v}(A,H)$ of $\underline{Y}_{v}(A,H)$ along the 256 $\P^3$s. 
The fixed locus of the induced involution on $\overline{Y}_{v}(A,H)$  is  smooth and four dimensional, 
hence the blow up $\widehat{Y}_{v}(A,H)$ of $\overline{Y}_{v}(A,H)$ along this fixed locus carries an 
involution $\widehat{\tau}_{v}$ admitting a smooth quotient  $\widehat{Y}_{v}(A,H)\slash \widehat{\tau}_{v}$.
This quotient is  $\widehat{K}_{v}(A,H)$ and  $\widehat{Y}_{v}(A,H)$ is its double cover branched over $\widehat{\Sigma}_{v}$. Finally 
$\widehat{K}_{v}(A,H)$ is  the blow up of $\wt{K}_{v}(A,H)$ along $256$ smooth $3$--dimensional quadrics.

This construction allows to relate the Hodge numbers of $\wt{K}_{v}(A,H)$ to the  
invariant Hodge numbers of $\underline{Y}_{v}(A,H)$. Finally, the  invariant Hodge numbers of 
$\underline{Y}_{v}(A,H)$ may be determined by using monodromy results of Markman \cite{mar}. 
This  yields  our main   result:

\begin{thm} Let $\wt K$ be an irreducible holomorphic symplectic of type OG6. The odd Betti numbers of $\wt K$ are zero, and its non--zero Hodge numbers are collected in the following table:
\[
\begin{array}{ccccccl}
    &   &    &    H^{0,0}=1 &  & & \\
   &   &  H^{2,0}=1  &    H^{1,1}=6 & H^{0,2}=1 & & \\
     & H^{4,0}=1 &  H^{3,1}=12  &  H^{2,2}=173 & H^{1,3}=12 & H^{0,4}=1 & \\
  H^{6,0}=1 & H^{5,1}=6 & H^{4,2}=173 &  H^{3,3}=1144 & H^{2,4}=173 & H^{1,5}=6 & H^{0,6}=1 \\
    & H^{6,2}=1 &  H^{5,3}=12  &  H^{4,4}=173 & H^{3,5}=12 & H^{2,6}=1 & \\
   &   &  H^{6,4}=1  &    H^{5,5}=6 & H^{4,6}=1 & & \\
   &   &    &   H^{6,6}=1 .&  & & \\
\end{array}
\]
\end{thm}

As a corollary, we also get the Chern numbers of this sixfold, see Proposition \ref{prop:chern} for details.

We should point out that this construction \emph{cannot} be carried out for IHS manifolds of type OG10, 
since the exceptional divisor of the second blow up 
(the procedure to obtain the symplectic resolution is the same) is \emph{not} divisible by $2$ in the integral cohomology.

The structure of the paper is as follows. In Section \ref{sec:reso}, 
we recall local and global properties of O'Grady's and Lehn--Sorger symplectic resolution. 
In Section \ref{local}, we construct an affine double of  the Lehn--Sorger local model of the 
deepest stratum of the singularity of $K_{v}(A,H)$, branched over the singular locus. In Section \ref{global}, we  globalize the previous results 
to construct global double covers $Y_{v}$ of $K_{v}(A,H)$ branched  over the singular locus.
In Section \ref{birgeom}, we prove that $Y_{v}$ is birational 
to an IHS manifold of K3$^{[3]}$--type. Finally, in Section \ref{sec:hodge}, 
we use the previous results to compute the Hodge numbers.

\bigskip
\subsection*{Notations}
For a closed embedding $X_1\subset X_2$ of algebraic  algebraic varieties  
we denote with $Bl_{X_1}X_2$ the blowup of $X_2$ along $X_1$.\\
For any affine cone or vector bundle  $X_3$,  we denote with $\mathbb{P}(X_3)$ its projectification.\\ 
Finally we denote by $H^{k}(X_{1})$ the $k$--th singular cohomology group of $X_1$ with rational coefficient.
 
\subsection*{Acknowledgments} We wish to thank Kieran O'Grady for useful discussions. 
The first two named authors are supported by FIRB 2012 ``Spazi di moduli ed applicazioni''.


\section{The resolution}\label{sec:reso}

Let us fix  a primitive Mukai vector $v_0 \in H^*_{alg}(A, \Z)$ with $v_0^2=2$,  set
$v=2v_0$, and consider a v-generic ample line bundle  $H$ on $A$ (see Section 2.1 of
\cite{per_rap}). 
By \cite[Th\'eor\`eme  1.1]{ls} the projective variety $\Kv:=K_v(A, H)$ admits a
simplectic resolution $\Kvt$ which is deformation equivalent to O'Grady's six dimensional example by \cite[Theorem 1.6(2)]{per_rap}. In this section we recall the description of the
singularity of $\Kv$ and of the symplectic resolution $f: \Kvt \to \Kv$ following both the papers of O'Grady
\cite{OG1}, \cite{OG1}
and Lehn and Sorger \cite{ls}.

Since the singular locus $\Sigma_{v}$ of $\Kv$ parametrizes polystable sheaves of the form
$F_1 \oplus F_2$, with $F_{i}\in M_{v_{0}}(X, H)$, we have 
$\Sigma_{v}=\Kv\cap \Sym^2 M_{v_0}(A, H)$. Since  $v_0^2=2$ the smooth moduli space 
$M_{v_{0}}$ is isomorphic to $A\times A^{\vee}$ and, as the Albanese map $\alb$ is an isotrivial
fibration,
the singular locus $\Sigma_{v}$ is isomorphic to $(A \times A^\vee) \slash \pm 1$.
This also implies that the singular locus $\Omega_{v}$ of $\Sigma_{v}$ consists of
$256$
points representing sheaves of the form $F^{\oplus2}$ with $F\in M_{v_{0}}(X, H)$.

The analytic type of the singularities appearing in  $\Kv$ is completely understood.
If $p\in \Sigma_{v}\setminus \Omega_v$, i.e. $p$ represents a polystable sheaf of
the form $F_1\oplus F_2$ where $F_1\ne F_2$, 
there exists a neighborhood $U\subset \Kv$ of $p$,  in the classical topology,
biholomorphic to a neighborhood of the origin in the
hypersurface defined in 
$\mathbb{A}^{7}$ by the equation $\sum_{i=1}^{3}x_{i}^{2}=0$ (see for example \cite[Prop. 4.4]{argiulia} or
\cite[Prop. 1.4.1]{OG1}), 
i.e. $\Kv$ has an $A_{1}$ singularity along $\Sigma_{v}\setminus
\Omega_v$.   

If $p\in\Omega_v$, the description of the analytic type of the singularity of $\Kv$ at
$p$ is due to Lehn and Sorger and it is contained in
\cite[Th\'eor\`eme 4.5.]{ls}. To recall this description, let $V$ be a four
dimensional vector space, let $\sigma$ be a symplectic form on $V$, and let
$\mathfrak{sp}(V)$ be the symplectic Lie algebra of $(V,\sigma)$, i.e. the Lie
algebra  of the Lie group of the automorphisms of $V$  preserving  the symplectic
form $\sigma$.
 
We let
\[
Z:=\{ A \in  \mathfrak{sp}(V) \,\, | \,\, A^2=0 \}
\]
be the subvariety  of matrices in $\mathfrak{sp}(V)$ having square zero. It is known
that $Z$ is the closure of the nilpotent orbit of type $\mathfrak o(2,2)$, which
parametrizes rank $2$ square zero matrices. Moreover, by  Criterion $2$ of \cite{hes}, $Z$ is also a normal variety.

By \cite[Th\'eor\`em 4.5.]{ls}, if $p\in\Omega_v$, there exists an euclidean neighborhood of $p$ in $\Kv$,
biholomorphic to a neighborhood of the origin in $Z$.
Hence the local geometry of a symplectic desingularization of $K_v$ is encoded in
the local geometry of a symplectic desingularization of $Z$.

Let $\Sigma$ be the singular locus of $Z$ and let $\Omega$ be the singular locus
of $\Sigma$.
Let us recall that $\dim Z=6$, $\dim \Sigma=4$, $\dim \Omega=0$ and,
more precisely,
\[
\Sigma=\{ A \in Z \,\, | \rk A \le 1 \, \}, \quad \text{ and } \quad \Omega=
\{ 0\}.
\]
Let $G \subset Gr(2, V) \subset \mathbb P (\wedge^2 V)$ be the Grassmannian of
Lagrangian subspaces of $V$, notice that $G$ is a smooth $3$-dimensional quadric and
set 
\[
 \Zt:=\{\,  (A, U) \,\, | \,\, A (U) =0 \, \} \subset Z \times G.
\]
The restriction $\pi_{G}:\Zt\rightarrow G$ of the second projection of $Z \times G$
makes $\Zt$ the total space of a $3$-dimensional vector bundle, the cotangent bundle of $G$.
In particular, $\Zt$ is smooth and the restriction
\[
f:\Zt\rightarrow Z
\]
of the
first projection of $Z \times G$, which is an isomorphism when restricted to the locus of rank $2$ matrices, is a resolution of the singularities.
The fiber $f^{-1}(A)$, over a point $A\in\Sigma$, is the smooth $\mathbb{P}^1$ parametrizing
Lagrangian subspaces contained in the $3$-dimensional kernel of $A$ and 
the central  fiber  $f^{-1}(0)$ is the whole $G$. As $Z$ has a $A_{1}$
singularity along $\Sigma\setminus\Omega$ and $G$ has dimension $3$ it
follows that        
$f:\Zt\rightarrow Z$ is a symplectic resolution.

\begin{rem}\label{R1}
Let $\mc U \subset V \otimes \mc O_G$ be the rank $2$ tautological bundle. The
smooth symplectic variety 
$\Zt$ is isomorphic to the total space  $\mathbf{Sym}^2_G \mc U$ of the second
symmetric power $\Sym^2_G \mc U$ of $\mc U$.
In fact, an endomorphism $ A \in  \mathfrak{gl}(V)$ belongs to $Z$
if and only if  the following conditions hold:
\begin{enumerate}
 \item{$A^2=0$,} 
\item{$\sigma(Av_{1},v_{2})=\sigma(Av_{2},v_{1})$ for any $v_{1},v_{2}\in V$.}
\end{enumerate}
By $(2)$ the kernel $\ker A$ and the image $\im A$ of  $A$ are orthogonal with respect to $\sigma$. Hence,  
for $(A,U) \in \Zt$, we have $V \to \im A \subset U \subset \ker A \subset V$.
Since $U \subset V$ is lagrangian we have $V \slash U \cong U^\vee$, so $A$ has a factorization of the form
$V \twoheadrightarrow U^\vee \to U \hookrightarrow V $.
Moreover  the induced linear map  
$\varphi_{A}\in \Hom (U^\vee,U)=U\otimes U$ defines a bilinear form on $U^\vee$
that is  symmetric if and only if  $(2)$ holds. 
\end{rem}

\begin{rem}\label{R2}
Set 
\[ \widetilde{\Sigma}:=\{\,  (A, U)\in \Zt \,\, | \,\, \rm{rank}(A)  \le 1 \, \}.
\]
The variety $\widetilde{\Sigma}$ is the exceptional locus of $f$. 
It is a locally trivial bundle over $G$ with fiber the affine cone over a conic in
$\mathbb{P}^{2}$.
Using the isomorphism $\Zt=\mathbf{Sym}^2_G \mc U$,  the variety
$\widetilde{\Sigma}$ is identified with the locus parametrizing singular symmetric bilinear forms
on the fibers of the dual of the tautological 
rank-2 vector bundle $\mc U$. In particular, $\wt \Sigma$ is a fibration over $G$ in cones over a smooth conic, i.e, $\widetilde{\Sigma}$ is singular
only along the zero-section $\widetilde{\Omega}\simeq G$ of $\mathbf{Sym}^2_G \mc U$ and it has an  
$A_{1}$ singularity along it.
\end{rem}

The following theorem due to Lehn and Sorger (\cite{ls}) gives an intrinsic
reformulation of the symplectic desingularization $f:\Zt\rightarrow Z$.

\begin{thm}[ \cite{ls} ]   \label{lehn sorger} Let $p \in \Omega$ be a singular point of
the singular locus $ \Sigma \subset \Kv$. Then,
\begin{enumerate}
\item[a)] (\cite[Th\'eor\`eme 4.5]{ls})
There is a local analytic isomorphism
\[
(Z, 0) \stackrel{loc}{\cong}(K_v, p).
\]
\item[b)] (\cite[Th\'eor\`eme 3.1]{ls}) The resolution $f:\Zt \rightarrow Z$, defined above,
coincides with the blowup of $Z$ along its singular locus $\Sigma$.
\end{enumerate}
\end{thm}

In order to discuss the topology of the symplectic desingularization of 
$f_{v}:\widetilde{K}_{v}\rightarrow \Kv$, we are going  to describe $f$ in term of blow ups
along smooth subvarieties.

\begin{prop} \label{local fundamental diagram} 
Let $\overline{\Sigma}$ be the strict transform of $\Sigma$ in $Bl_{\Omega}
Z$. 
\begin{enumerate}
\item{$\overline{\Sigma}$ is the singular locus of $Bl_{\Omega}$,
$\overline{\Sigma}$ is smooth and $Bl_{\Omega} Z$ has an $A_{1}$
singularity along $\overline{\Sigma}$.} 
\item{The varieties $Bl_{\widetilde{\Omega}}Bl_{\Sigma} Z$ and
$Bl_{\overline{\Sigma}}  Bl_{\Omega} Z$ are smooth and  isomorphic over $Z$.
In particular, the diagram  
\[
\xymatrix{
&   Bl_{\widetilde{\Omega}}Bl_{\Sigma} Z=Bl_{\overline{\Sigma}} 
Bl_{\Omega} Z   \ar[dr]^\xi \ar[dl]_\rho & \\
\widetilde{Z}=Bl_{\Sigma} Z \ar[dr]_{f} & & Bl_{\Omega} Z\; , 
\ar[dl]^\eta  \\
 & Z&  
}
\]
where the arrows are blow up maps, is commutative.}   
\end{enumerate}
\end{prop}
\begin{proof}(1) Let $\mathbb{P}(Z):=Z/{\mathbb{C}^{*}}$ be the projectivization of
the affine cone Z. 
As $Z$ is a cone, its blow up $Bl_{\Omega} Z$ at the origin is the total space
of the tautological line bundle   over $\mathbb{P}(Z)$.
The singular locus  $\Sigma$ of $Z$ is a subcone, hence its strict transform
$\overline{\Sigma}= Bl_{\Omega} \Sigma$ is the total space of the restriction to $\mathbb{P}(\Sigma) \subset \mathbb{P}(Z)$ of the tautological line bundle. 
As $\Sigma\setminus \{0\}$ is smooth, $\mathbb{P}(\Sigma)$ is smooth. Moreover,  since $Z$ has an $A_{1}$
singularity along $\Sigma\setminus \{0\}$, the singular locus of $\mathbb{P}(Z)$ is  
$\mathbb{P}(\Sigma)$ and $\mathbb{P}(Z)$ has an $A_{1}$
singularity along $\mathbb{P}(\Sigma)$. 
Passing to the total spaces of the tautological line bundles we get item (1).

(2)  We only need to show that  $Bl_{\widetilde{\Omega}}Bl_{\Sigma} Z$ and
$Bl_{\overline{\Sigma}}  Bl_0 Z$
are isomorphic. 
By Remark \ref{R1}, $\widetilde{Z}$ is isomorphic to $\mathbf{Sym}^2_G \mc U$ and,
by Remark \ref{R2}, $Bl_{\widetilde{\Omega}}Bl_{\Sigma} Z$ is the blow up of  
$\mathbf{Sym}^2_G \mc U$ along its zero section. Letting $\mathbb{P}(\Sym^2_G \mc
U)$ be the projective bundle associated to 
$\Sym^2_G \mc U$, the blow up $Bl_{\widetilde{\Omega}}Bl_{\Sigma} Z$ is
isomorphic to the total space 
$T\subset \mathbb{P}(\Sym^2_G \mc U)\times_{G} \mathbf{Sym}^2_G \mc U$
of the tautological line bundle  of the projective bundle $\mathbb{P}(\Sym^2_G \mc U)$ .
The isomorphism  $$\mathbf{Sym}^2_G \mc U=\widetilde{Z}:=\{\,  (A, U) \,\, | \,\, A
(U) =0 \, \} \subset Z \times G$$ also implies  
$$\mathbb{P}(\Sym^2_G \mc U)= \{\,  ( [A], U) \,\, | \,\, A (U) =0 \, \}
\subset \mathbb{P}(Z) \times G$$ and, using this identification,
we conclude that
$$Bl_{\widetilde{\Omega}}Bl_{\Sigma} Z=T= \{\,  ( [A], B, U) \,\, |
\,\, A (U) =0 \,\, B\in [A]\}\subset \mathbb{P}(Z) \times Z\times G.$$

On the other side of the diagram, as $Z$ is a cone, its blow up  at the origin 
can  explicitly be given as 
$$Bl_{\Omega} Z= \{\,  ( [A], B) \,\, | \,\,  B\in [A]\}\subset
\mathbb{P}(Z) \times Z.$$
It remains to show that the map $\xi:Bl_{\widetilde{\Omega}}Bl_{\Sigma}
Z\rightarrow  Bl_{\Omega} Z$ induced by the projection 
$\pi_{1,2}:\mathbb{P}(Z) \times Z\times G\rightarrow \mathbb{P}(Z) \times Z$ is the
blow up of $\mathbb{P}(Z) \times Z$ along $\overline{\Sigma}$.
Since for $q\in \overline{\Sigma}$ the schematic fiber $\xi^{-1}(q)$ is
isomorphic to $\mathbb{P}^{1}$ and $\overline{\Sigma}$ is smooth, the schematic inverse image
$\xi^{-1}(\overline{\Sigma})$ is a smooth, hence reduced and irreducible Cartier divisor. By the
universal property of blow ups, $\xi$ factors through a proper map 
$\iota: Bl_{\widetilde{\Omega}}Bl_{\Sigma} Z\rightarrow
Bl_{\overline{\Sigma}}  Bl_{\Omega} Z $ sending
$\xi^{-1}(\overline{\Sigma})$
surjectively onto the exceptional divisor of the blow up  of
$Bl_{\Omega} Z$ along $\overline{\Sigma}$.
Finally, since $Bl_{\Omega} Z$ is only singular along   $\overline{\Sigma}$
and has an $A_{1}$ singularity along $\overline{\Sigma}$, the blow up  
$Bl_{\overline{\Sigma}}  Bl_{\Omega} Z$ is smooth. It follows that  $\iota$ is a proper
birational map between smooth varieties that does not contract any divisor,
therefore $\iota$ is a isomorphism.
\end{proof}

This proposition also allows us to describe the exceptional loci of the blow up maps
appearing in item (2).

Let  $\widehat{\Sigma}\subset Bl_{\widetilde{\Omega}}Bl_{\Sigma} Z$ be the
exceptional divisor of $\xi$, let $\widehat{\Omega}\subset Bl_{\widetilde{\Omega}}Bl_{\Sigma} Z$ be the
exceptional divisor of $\rho$, and recall that $\wt \Omega \cong G$ is the inverse image of $\Omega$ under the resolution $f$.
\begin{cor} \label{corlfd}
\begin{enumerate}
\item{$\widehat{\Sigma}$ is a $\mathbb{P}^{1}$-bundle over $\overline{\Sigma}$
and $\widehat{\Sigma}=Bl_{\widetilde{\Omega}}\widetilde{\Sigma}$.} 
\item{$\widehat{\Omega}$ is a $\mathbb{P}^{2}$-bundle over $\widetilde{\Omega}$ isomorphic to  $\mathbb{P}(Sym^2_G \mc U)$.
} 
\end{enumerate}
\end{cor}
\begin{proof}
(1) By item (1) of Proposition \ref{local fundamental diagram}, the restriction of
$\xi$ realizes  $\widehat{\Sigma}$ as a $\mathbb{P}^{1}$-bundle over
$\overline{\Sigma}$.
Since $\widehat{\Sigma}$ is also the strict transform of $\widetilde{\Sigma}$ under $\rho$
and $\widetilde{\Omega}\subset \widetilde{\Sigma}$, the restriction of $\rho$ to
 $\widehat{\Sigma}$ can be identified with the blow up map of
$\widetilde{\Sigma}$ along  $\widetilde{\Omega}$. As for
(2), we can argue as follows. Since $\widetilde{\Omega}$ is a smooth subvariety of codimension $3$ in the
smooth variety $\Zt$ the restriction of $\rho$
to $\widehat{\Omega}$ makes it a $\mathbb{P}^{2}$-bundle over
$\widetilde{\Omega}$. More precisely, since $\widetilde{\Omega}$ is the zero
section of 
 $\widetilde{Z}= \mathbf{Sym}^2_G \mc U$ (see Remark \ref{R2}),  there is an isomorphism 
 $\widehat{\Omega}\simeq \mathbb{P}(Sym^2_G \mc U)$.
\end{proof}

To compute invariants of $\Kvt$ we need the following global versions of Proposition
\ref{local fundamental diagram} and Corollary \ref{corlfd}.

\begin{prop} \label{fundamental diagram} 
Let $\overline{\Sigma}_{v}$ be the strict transform of $\Sigma_{v}$ in
$Bl_{\Omega_{v}} \Kv$. 
\begin{enumerate}
\item{
$\overline{\Sigma}_{v}$ is the singular locus of $Bl_{\Omega_{v}} \Kv$, 
$\overline{\Sigma}_{v}$ is smooth and $Bl_{\Omega_{\Kv}} \Kv$ has an $A_{1}$
singularity along $\overline{\Sigma}_{v}$.} 
\item{The projective varieties $Bl_{\widetilde{\Omega}_{v}}Bl_{\Sigma_{v}} \Kv$ and
$Bl_{\overline{\Sigma}_{v}}  Bl_{\Omega_{v}} \Kv$ are smooth and isomorphic over $\Kv$.
Hence the diagram  
\[ 
\xymatrix{
&   Bl_{\widetilde{\Omega}_{v}}Bl_{\Sigma_{v}} \Kv=Bl_{\overline{\Sigma}_v} 
Bl_{\Omega_{v}} \Kv   \ar[dr]^{\xi_{v}} \ar[dl]_{\rho_{v}} & \\
\widetilde{K}_{v}=Bl_{\Sigma_{v}} \Kv \ar[dr]_{f_{v}} & & Bl_{\Omega_{v}} \Kv\; , 
\ar[dl]^{\eta_{v}}  \\
 & \Kv&  
}
\]
where the arrows are blow ups, is commutative.}
\end{enumerate}   
\end{prop}
\begin{proof}
As $\Sigma_{v}\setminus\Omega_{v}$ is smooth and $\Kv$ has an $A_{1}$ singularity
along $\Sigma_{v}\setminus\Omega_{v}$, item (1) follows from  
Theorem \ref{lehn sorger}(a) and Proposition \ref{local fundamental diagram}(1),
since the blow up is a local construction. 
Item (2) holds since item (2) of Proposition \ref{local fundamental diagram} also
implies that the natural birational map between
$Bl_{\widetilde{\Omega}_{v}}Bl_{\Sigma_{v}} \Kv$ and $Bl_{\overline{\Sigma}_{v}} 
Bl_{\Omega_{v}} \Kv$ is actually an isomorphism.
\end{proof}

\begin{rem} Since $\Sigma_{v}$ contains $\Omega_{v}$ as a closed subscheme, its strict
transform  $\overline{\Sigma}_{v}$ in $Bl_{\Omega_{v}}K_{v}$ is 
 isomorphic to the blow up $Bl_{\Omega_{v}}\Sigma_{v}$. Recall that  
$\Sigma_{v}\simeq (A\times A^{\vee})\slash \pm 1$, so that its singular locus $\rm{Sing}((A\times A^{\vee})\slash \pm 1)$
is in bijective correspondence with the set of $2$-torsion points $(A\times A^{\vee})[2]$ of $A\times A^{\vee}$. It follows that  
there is a chain of  isomorphisms
$$\overline{\Sigma}_{v}\simeq Bl_{\rm{Sing}((A\times A^{\vee})\slash \pm 1)}((A\times A^{\vee})\slash
\pm 1)\simeq (Bl_{(A\times A^{\vee})[2]}(A\times A^{\vee}))\slash \pm 1.$$ 
This also implies that the exceptional divisor of $Bl_{\Omega_{v}}\Sigma_{v}$, which is given by the (reduced induced) intersection
of the exceptional divisor $\overline{\Omega}_{v}$ of $Bl_{\Omega_{v}}\Kv$ and $\overline{\Sigma}_{v}$, 
consists of a union of $256$ disjoint $\mathbb{P}^{3}$.    
\end{rem}

\begin{cor} \label{corfd}
Let  $\widehat{\Sigma}_{v}\subset Bl_{\widetilde{\Omega}_{v}}Bl_{\Sigma_{v}} \Kv$ be the
exceptional divisor of $\xi$, let $\widehat{\Omega}_{v}\subset Bl_{\widetilde{\Omega}_{v}}Bl_{\Sigma_{v}} \Kvt$ be the
exceptional divisor of $\rho$, let $\overline{\Omega}_{v}\subset Bl_{\Omega_{v}} \Kv$ be the exceptional divisor
of $\eta$, and, finally, let $\overline{\Omega}_{v}\cap \overline{\Sigma}_{v}$ denote the
intersection
of $\overline{\Omega}_{v}$ and $\overline{\Sigma}_{v}$ with its reduced induced structure.

\begin{enumerate}
\item{$\widehat{\Sigma}_{v}$ is a $\mathbb{P}^{1}$-bundle over $\overline{\Sigma}_{v}$
and $\widehat{\Sigma}_{v}=Bl_{\widetilde{\Omega}_{v}}\widetilde{\Sigma}_{v}$.} 
\item{$\widehat{\Omega}_{v}$ is a $\mathbb{P}^{2}$-bundle over $\widetilde{\Omega}_{v}$ isomorphic to
$\mathbb{P}(Sym^2_G \mc U)$}
\end{enumerate}
\end{cor}
\begin{proof} This follows from item (2) of Theorem \ref{lehn sorger} and Corollary \ref{corlfd}.
\end{proof}
\begin{rem}
 The proof of the existence of an isomorphism between the  smooth projective varieties $Bl_{\widetilde{\Omega}_{v}}Bl_{\Sigma_{v}} \Kv$ and
$Bl_{\overline{\Sigma}_{v}}  Bl_{\Omega_{v}} \Kv$ follows the original strategy used by O'Grady in \cite{OG1}.
For $v=(2,0,-2)$, he proved that a symplectic desingularization of $\Kv$ can be obtained by  contracting
the strict transform  $\widehat{\Omega}_{v}$ of $\overline{\Omega}_{v}$ in  $Bl_{\overline{\Sigma}_{v}}  Bl_{\Omega_{v}} \Kv$.
Proposition \ref{fundamental diagram} shows, in particular, that O'Grady's procedure gives a symplectic 
desingularization of $\Kv$ that is isomorphic to the Lehn-Sorger desingularization $Bl_{\Sigma_{v}} \Kv$.
The proof of  Proposition \ref{fundamental diagram} is elementary because it uses the crucial description, due to Lehn and Sorger, of the analytic type of the singularities appearing in $\Kv$.
\end{rem}


\section{The local covering} \label{local}

This section is devoted to the local description of the double cover,  branched along the singular locus, of O'Grady's singularity.

It is known \cite[Cor. 6.1.6]{col} that the fundamental group of the open orbit
$\mathfrak o(2,2)$ is isomorphic to $\Z/(2)$. We wish to extend this double cover to
a ramified double cover of $\overline {\mathfrak o(2,2)}=Z$.

To this aim, let
\[
W :=\{ \,\, v \otimes w \,\, | \,\,\sigma(v, w)=0 \, \} \subset V \otimes V, \quad
\text{ and }  \quad \Delta_W=\{ \,\, v \otimes v\,\, \} \subset W.
\]
be the affine cone over the incidence subvariety
\[
I:=\{ ([v], [w]) \,\, | \,\, \sigma(v, w)=0 \, \} \subset \P V \times \P V\subset \P(V\otimes V).
\]
Since $I$ is smooth, the singular locus $\Gamma$ of $W$  consists only of the vertex $0\in V \otimes V$. 

Moreover, 
since $I\subset \P(V\otimes V)$ is projectively normal, $W$
is a normal variety.

Let
\[
\tau: W \to W
\]
be the involution induced by restricting the linear involution $\tau_{V\otimes V}$ on $V\otimes V$
 that interchanges the two factors.

The following lemma exhibits $W$ as the desired double cover of $Z$.
\begin{lem} \label{W e Z}
 The morphism
\[
\begin{aligned}
\varepsilon: W & \longrightarrow Z \\
 v \otimes w & \longmapsto  \sigma(v, \cdot ) w+ \sigma(w, \cdot) v
\end{aligned}
\]
realizes $Z$ as the quotient $W \slash \tau$. In particular, $\varepsilon$ is a
finite $2:1$ morphism, the ramification locus of $\varepsilon$ is $\Delta$ and the branch locus of $\varepsilon$ is  $\Sigma$.
\end{lem}
\begin{proof} 
We leave it to the reader to check that $\varepsilon (W)\subset Z$.
For a rank $2$ endomorphism $A \in Z\setminus \Sigma$, let us show that $\varepsilon^{-1}(A)$ consists of $2$ points interchanged by $\tau$.
Let $U\subset V$ be the kernel of $A$, which is a Lagrangian subspace.
As shown in Remark \ref{R1}, $A$ induces a linear map $\varphi_{A}\in \Hom(U^{\vee},U)=U\otimes U$ that gives a rank $2$ bilinear symmetric form
on $U^{\vee}$ and, conversely, any symmetric bilinear form on $U^{\vee}$ determines a a rank $2$ endomorphism $A\in Z$ whose kernel is $U$.
A rank $2$ symmetric bilinear form on $U^{\vee}$ is determined, up to scalars, by  $2$ independent distinct isotropic vectors $L_{1}$ and $L_{2}$, hence by their kernels $\ker(L_1)\subset U$ and $\ker(L_2)\subset U$. Now it suffices  to notice that, for $v$ and $w$ spanning $U$ and for $A=\varepsilon (v\otimes w)$,
the lines $\ker(L_1)$ and  $\ker(L_2)$ are the lines generated by $v$ and $w$.

 Since $\ker A$ and $\im A$ are orthogonal (see Remark
\ref{R1}), if $A\in \Sigma$ is a rank $1$ endomorphism or the $0$ endomorphism, then there exists a unique up to scalars $v\in V$ such that $A=\sigma(v, \cdot ) v$. This shows that $\varepsilon^{-1}(A)$ consists  of a unique point, which is fixed by $\tau$.

 To show that $Z \cong W \slash
\tau$, notice that $\varepsilon$ is $\tau$--invariant and its fibers are the orbits of the action of $\tau$, hence $\varepsilon$ induces a bijective
morphism $W \slash \tau \to Z$ . Since $Z$ is normal (see Criterion $2$ of \cite{hes}), this morphism is an
isomorphism. 
\end{proof}
\begin{rem}\label{ridimo}
Using  Lemma \ref{W e Z}, we may reprove that the fundamental group of $\mathfrak o(2,2)= Z\setminus \Sigma$ is isomorphic to $\Z/(2)$.
As $\varepsilon$ is \'etale on $Z\setminus \Sigma$,  it suffices to show that $\varepsilon ^{-1}(Z\setminus \Sigma)=W\setminus \Delta$ is simply connected.
$W\setminus \Delta$ can be obtained from the smooth variety $W\setminus \{0\}$ by removing a codimension $2$ subvariety, 
hence there is an isomorphism of fundamental
groups $\pi_1(W\setminus \Delta)\simeq \pi_1(W\setminus \{0\})$. Finally the map $k:W\setminus \{0\}\rightarrow \mathbb{P}(V)$ defined by $k(v\otimes w)=[v]$
is a locally trivial fibration with fiber isomorphic to the complement of $0$ in a $3$ dimensional vector space. Therefore $k$ has simply 
connected base and fiber and $\pi_1(W\setminus \{0\})=0$.    
\end{rem}

The morphism $\varepsilon$ induces double coverings of the varieties  $\widetilde{Z}=Bl_{\Sigma} Z$, 
  $Bl_{\Omega} Z$ and $Bl_{\widetilde{\Omega}}Bl_{\Sigma} Z = Bl_{\overline{\Sigma}}Bl_{\Omega}$.
The following corollary discusses the case of $Bl_{\Omega} Z$.
\begin{cor}\label{overlineepsilon}
The morphism $\varepsilon$ lifts to a finite $2:1$ morphism
\[
\overline{\varepsilon}: Bl_{\Gamma} W \to Bl_{\Omega} Z,
\]
whose branch  locus is the strict transform  $\overline{\Sigma}$  of $\Sigma$ in $Bl_{\Omega} Z$. 
\end{cor}
\begin{proof}
The morphism $\varepsilon$ is the restriction to $W$ of the linear map  \[\varepsilon_{V\otimes V}:V\otimes V\rightarrow \mathfrak{sp}(V)\]
sending $v\otimes w$ to $\sigma(v, \cdot ) w+ \sigma(w, \cdot) v$ for any $v\otimes w\subset V\otimes V$.
As $\ker \varepsilon_{V\otimes V}\cap W=0$, the map $\epsilon$ induces a morphism $\P(\varepsilon): I\rightarrow \P(Z)$ between the projectivization of
$W$ and $Z$. There are the identifications
\[ Bl_{\Gamma} W=Bl_{0}W=\{\,  (\mathbb{C}\alpha, v\otimes w)\in I \times W \,\, | \,\, v\otimes w\in \mathbb{C}\alpha\, \}  \]
and 
\[Bl_{\Omega}Z=Bl_{0}Z=\{\,  (\mathbb{C}A, B)\in \mathbb{P}(Z)\times Z \,\, | \,\, B\in \mathbb{C}A\, \}  .\]

It follows that $\P(\varepsilon)\times \varepsilon$ restricts to a map
$\overline{\varepsilon}:Bl_\Gamma W \to Bl_\Omega Z$ and, by Lemma \ref{W e Z},
$\overline{\varepsilon}$ is a finite $2:1$ map whose  branch locus is
\[\overline{\Sigma}_W=Bl_{\Omega}\Sigma=Bl_{0}\Sigma=
\{\,  (\mathbb{C}A, B\in \mathbb{P}(\Sigma)\times \Sigma \,\, | \,\, B\in \mathbb{C}A\, \}.  \]
\end{proof}
\begin{rem}\label{lisci su} Since $W$ is the cone over a smooth variety, both its blow up at the origin $Bl_{\Gamma}W$ and the exceptional divisor $\overline{\Gamma}\subset Bl_{\Gamma}W$ are smooth. 
 Finally the strict transform $\overline{\Delta}$ of $\Delta$ in $Bl_{\Gamma}W$ is isomorphic to 
$Bl_{\Gamma}\Delta$ and, since   
$\Delta$ is the cone over a smooth variety,
also $\overline{\Delta}$ is smooth.
\end{rem}
The following corollary treats the case of the induced double cover of $\Zt$

\begin{cor}\label{lift}
Let $\pi:SC_{G}\mathcal{U}^{\otimes 2}\rightarrow G$ be the relative affine Segre cone parametrizing 
decomposable tensors in the total space of the rank $4$ vector bundle $\mathcal{U}^{\otimes 2}$.
\begin{enumerate}
\item{$SC_{G}\mathcal{U}^{\otimes 2}$ is isomorphic to $Bl_{\Delta}W$,}

\item{Using this identification, the map \[\widetilde{\varepsilon}:Bl_{\Delta}W(=SC_{G}\mathcal{U}^{\otimes 2})\rightarrow \Zt(=\mathbf{Sym}^2_G \mc U),\]   
induced by symmetrization on the fibers, is a finite $2:1$ morphism lifting $\varepsilon$, whose branch locus is $\widetilde{\Sigma}$.}
\end{enumerate}
\end{cor}
\begin{proof}
(1) By definition of fiber product, $W\times_{Z} \Zt$ is equal to
\[
\{(v\otimes w, A, U)\in W \times Z\times G  \,\, | \,\, \varepsilon (v \otimes w)=A\,\, \text{and }\,\, v,w\in U\,\}
\]
and, by Lemma \ref{W e Z}, the fiber over $U$ of the projection $\pi_{G}:W\times_{Z}\Zt\rightarrow G$ is naturally isomorphic to
the variety  $SC\;U^{\otimes 2 }$ of decomposable tensors in $U\otimes U$. It follows that $SC_{G}\mathcal{U}^{\otimes 2}$ is isomorphic to $W\times_{Z} \Zt$.

Let us show that $W\times_{Z} \Zt$ has a birational morphism to $Bl_{\Delta}W$. Let $\pi_{W}: W\times_{Z} \Zt\rightarrow W$ be the projection,
by the universal property of blow ups, it will suffice to show that the schematic inverse image $\pi_{W}^{-1}(\Delta)$ is a Cartier divisor.

For $U\in G$, the projection $\pi_{W}$ sends the fiber $\pi_{G}^{-1}(U)$ isomorphically onto  $SC\;U^{\otimes 2}$. 
Hence the schematic intersection $\pi_{W}^{-1}(\Delta)\cap \pi^{-1}_{G}(U)$ is isomorphic to
the schematic intersection $\Delta\cap SC\;U^{\otimes 2 }$, i.e. the reduced cone over a smooth conic $C\subset \mathbb{P}(U\otimes U)$  parametrizing symmetric
decomposable tensors in $U\otimes U$. As varying $U\in G$ the intersections $\pi_{W}^{-1}(\Delta)\cap \pi^{-1}_{G}(U)$ form a locally trivial family over $G$,
 the family $\pi_{W}^{-1}(\Delta)\rightarrow G$ is locally trivial. Finally, 
as the cone over $C\subset \mathbb{P}(U\otimes U)$ is a Cartier divisor in the variety of decomposable tensors of $U\otimes U$,
the scheme $\pi_{W}^{-1}(\Delta)$ is a Cartier divisor in $W\times_{Z} \Zt$.

On the other hand, $Bl_{\Delta}W$ has a regular birational morphism to  $W\times_{Z} \Zt$ inverting the previous birational  morphism.
   
This will follow if we prove that the ideal of $\varepsilon^{-1}(\Sigma)$ in $W$ is the square $I_{\Delta}^{2}$ 
of the ideal of $\Delta$. In fact, in this case, the blow up $Bl_{\Delta}W$ equals the blow up
$Bl_{\varepsilon^{-1}(\Sigma)}W$, hence the schematic inverse image of $\Sigma$ in $Bl_{\Delta}W$ is a Cartier divisor.
Therefore, as $\Zt=Bl_{\Sigma}Z$, by the universal property of blow ups, we can conclude that there exists a commutative diagram  
\[\xymatrix{
Bl_{\Delta}W  
\ar[d]_{g} \ar[r]^{\xi'}  & \Zt \ar[d]_{f}\\
   W\ar[r]_{\varepsilon} & Z \\
}
\] inducing the desired  birational regular morphism from 
$Bl_{\Delta}W$ to    $W\times_{Z} \Zt$.

To determine the ideals of $\varepsilon^{-1}(\Sigma)$ and $\Delta$ in $W$ we recall that the involution $\tau$
is the restriction of the  linear involution $\tau_{V\otimes V}$ on $V\otimes V$ that can be interpreted as the transposition 
on $4\times 4$ matrices if we chose a basis for $V$. Moreover, the ideal of $\Delta$ in $W$ is generated by the restrictions of the linear
antiinvariant functions on  $V\otimes V$ (this already holds for the ideal of $\Delta$ in the affine cone over the Segre variety $\mathbb{P}(V)\times \mathbb{P}(V)\subset \mathbb{P}(V\otimes V))$.  Hence, $I_{\Delta}^{2}$ is generated by restrictions of products of pairs of linear antiinvariant functions on $V\otimes V$ and any such product comes from a function on the quotient $(V\otimes V)/\tau_{V\otimes V}$ vanishing along the branch locus $B$. Since $B$ contains the branch locus $\Sigma$ of
$W/\tau$, we conclude that the ideal of $\varepsilon^{-1}(\Sigma)$   contains $I_{\Delta}^{2}$. Equality holds becouse $W\setminus\{0\}$ is smooth and the fixed locus 
$\Delta\setminus \{0\}$ has codimension $2$, hence  $\varepsilon^{-1}(\Sigma)$ equals the subscheme $\Delta_{2}$ defined by $I_{\Delta}^{2}$ outside the origin.
As $\Delta_{2}$ is a subcone of $W$, it is the closure of  $\Delta_{2}\setminus \{0\}$, therefore it is a closed subscheme of $\varepsilon^{-1}(\Sigma)$.

(2) The existence of the regular morphism $\widetilde{\varepsilon}$ lifting $\varepsilon$ follows from (1). The branch locus of $\tilde{\varepsilon}$ is  
$\widetilde{\Sigma}$ because, by our description of $\tilde{\varepsilon}$, it parametrizes singular bilinear symmetric tensors (see Remark \ref{R2}).
\end{proof}
Corollary \ref{lift} also allows us to describe the singularities of the exceptional divisor $\widetilde{\Delta}$ of the blow up $Bl_{\Delta}W$ of $W$ along $\Delta$.   

\begin{rem}\label{singtilde}
$Bl_{\Delta}W \simeq SC_{G}\mathcal{U}^{\otimes 2}$ is a locally trivial bundle over $G$ with fiber the affine cone over a smooth quadric in $\mathbb{P}^{3}$.
Hence it is smooth outside the zero section $\widetilde{\Gamma}$ and any point of $\widetilde{\Gamma}$ has a neighborhood isomorphic to the product of the affine cone over a smooth quadric and
a smooth $3$-dimensional variety.
As $\varepsilon$ and $\widetilde{\varepsilon}$ are finite, the morphism  $\widetilde{\varepsilon}$ sends the exceptional divisor $\widetilde{\Delta}\subset Bl_{\Delta}W$
onto the exceptional divisor $\widetilde{\Sigma}\subset\Zt$.
By the definition of $\widetilde{\varepsilon}$ in item (2) of Corollary \ref{lift}, the divisor $\widetilde{\Delta}$ parametrizes 
symmetric
decomposible tensors in the fibers of $\pi_{G}:SC_{G}\mathcal{U}^{\otimes 2}\rightarrow G$, 
hence it is a locally trivial bundle with fiber the affine cone over a smooth conic.
Therefore it is smooth outside $\widetilde{\Gamma}$ and has an $A_{1}$ singularity along $\widetilde{\Gamma}$.
\end{rem}

The following corollary completes the picture of the double covering induced by $\varepsilon$ in the local case.

\begin{cor} \label{diagram cover}
\begin{enumerate}
\item{There exist finite degree $2$  morphisms $\widehat{\varepsilon}_{1}: Bl_{\widetilde{\Gamma}}Bl_{\Delta}W \rightarrow Bl_{\widetilde{\Omega}}Bl_{\Sigma} Z$
and $\widehat{\varepsilon}_{2}: Bl_{\overline{\Delta}}Bl_{\Gamma}W \rightarrow Bl_{\overline{\Sigma}}Bl_{\Omega} Z$, 
lifting $\widetilde{\varepsilon}$ and   
 $\overline {\varepsilon}$, whose  branch loci are  the strict transform of $\widetilde{\Sigma}$ in $Bl_{\widetilde{\Omega}}Bl_{\Sigma} Z$
and the exceptional divisor $\widehat{\Sigma}$ of $Bl_{\overline{\Sigma}}Bl_{\Omega} Z$ respectively.}
\item{The varieties   $Bl_{\widetilde{\Gamma}}Bl_{\Delta}W$ and  $Bl_{\overline{\Delta}}Bl_{\Gamma}W$ are smooth and isomorphic over $W$. 
Hence there exists a commutative diagram 
\[
\xymatrix{
&   Bl_{\widetilde{\Gamma}}Bl_{\Delta}W  = Bl_{\overline{\Delta}}Bl_{\Gamma}W 
\ar[dr] \ar[d]_{\widehat{\varepsilon}_{1}=\widehat{\varepsilon}_{2}}  \ar[dl] & \\
Bl_{\Delta}W  \ar[d]_{\widetilde{\varepsilon}} \ar[dr] &  
Bl_{\widetilde{\Omega}}Bl_{\Sigma} Z = Bl_{\overline{\Sigma}}Bl_{\Omega} \ar[dl] \ar[dr] Z & Bl_{\Gamma} W \ar[d]_{\overline{\varepsilon}}
\ar[dl] 
\\
\Zt =Bl_{\Sigma}Z \ar[dr]& W \ar[d]_{\varepsilon} & Bl_{\Omega} Z\ar[dl]\\
& Z & \\
}
\]
where the diagonal arrows are blow ups.

}
\end{enumerate}

\end{cor}
\begin{proof}
(1) Recall that $Bl_{\Delta}W\simeq SC_{G}\mathcal{U}^{\otimes 2}$ and 
$Bl_{\Sigma}Z\simeq \mathbf{Sym}^2_G \mc U$ are locally trivial bundles  over $G$
and $\widetilde{\Gamma}$ and $\widetilde{\Omega}$ are their respective zero sections.
As $\widetilde{\varepsilon}: Bl_{\Delta}W\rightarrow Bl_{\Sigma}Z$ is a morphism over $G$, the existence 
of $\widehat{\epsilon}_{1}: Bl_{\widetilde{\Gamma}}Bl_{\Delta}W \rightarrow Bl_{\widetilde{\Omega}}Bl_{\Sigma} Z$
branched over the strict transform  of $\widetilde{\Sigma}$ in $Bl_{\widetilde{\Omega}}Bl_{\Sigma} Z$ 
follows from the existence of a commutative diagram of the form 
\[\xymatrix{
Bl_{0}SC U^{\otimes 2}  
\ar[d] \ar[r]  & Bl_{0} Sym^{2}U \ar[d]\\
  SC U^{\otimes 2} \ar[r] & Sym^{2}U \\
}
\]
where $U$ is a $2$-dimensional vector space, $SC \;U^{\otimes 2}\subset U\otimes U$ is the affine cone, parametrizing decomposable tensors, over the Segre variety 
$\mathbb{P}(U)\times \mathbb{P}(U)$, the vertical arrows are blow ups, 
the horizontal arrows are induced by symmetrization (hence their branch locus parametrizes singular symmetric tensors). 

By Corollary \ref{overlineepsilon}, the branch locus of the finite $2:1$ morphism 
$\overline{\varepsilon}: Bl_{\Gamma}W \rightarrow Bl_{\Omega} Z$
is the singular locus 
$\overline{\Sigma}$ of  $Bl_{\Omega} Z$. By item (1) of Proposition
\ref{local fundamental diagram}, $Bl_{\Omega} Z$ has an $A_{1}$ singularity along 
$\overline{\Sigma}$ and this suffices to imply the existence of the desired finite $2:1$ morphism 
$\widehat{\varepsilon}_{2}: Bl_{\overline{\Delta}}Bl_{\Gamma}W \rightarrow Bl_{\overline{\Sigma}}Bl_{\Omega} Z$
whose branch locus is the exceptional divisor $\widehat{\Sigma}$ of $Bl_{\overline{\Sigma}}Bl_{\Omega} Z$.

(2) Smoothness of $Bl_{\widetilde{\Gamma}}Bl_{\Delta}W$ and  
$Bl_{\overline{\Delta}}Bl_{\Gamma}W$ follow from Remark \ref{singtilde} and Remark \ref{lisci su} respectively.

It remains  to show that the natural birational map $j:Bl_{\widetilde{\Gamma}}Bl_{\Delta}W \dashrightarrow Bl_{\overline{\Delta}}Bl_{\Gamma}W$
extends to a biregular morphism. Using the identification $Bl_{\widetilde{\Omega}}Bl_{\Sigma} Z = Bl_{\overline{\Sigma}}Bl_{\Omega}$,
$\widehat{\varepsilon}_{1}$ and   $\widehat{\varepsilon}_{2}$ may be seen as finite covers of  $Bl_{\widetilde{\Omega}}Bl_{\Sigma} Z$ and 
we have an equality of rational maps
$\varepsilon_{2}\circ j= \varepsilon_{1}$. It follows that the closure of the graph of $j$ is contained in the fiber product
$Bl_{\widetilde{\Gamma}}Bl_{\Delta}W \times_{Bl_{\widetilde{\Omega}}Bl_{\Sigma} Z} Bl_{\overline{\Delta}}Bl_{\Gamma}W$.
As $\widehat{\varepsilon}_{1}$ and   $\widehat{\varepsilon}_{2}$ are finite, the closure of the graph of $j$ is finite and
generically injective on the smooth factors $Bl_{\widetilde{\Gamma}}Bl_{\Delta}W$ and $ Bl_{\overline{\Delta}}Bl_{\Gamma}W$.
By Zariski's Main Theorem it is the graph of an isomorphism extending $j$. 

The commutativity of the diagram holds because all maps are regular and commutativity is trivial on open dense subsets.

\end{proof}

In the final remark of this section we discuss the behavior of the restriction of the morphisms appearing in the  diagram in item (2) of Corollary \ref{diagram cover}, to the divisors appearing over $\Omega$. Since  this remark will not be used in the rest of the paper, some of the computations are left to the reader.

\begin{rem}
Let $\widehat{\Gamma}$ be the exceptional divisor of the blow up of $Bl_{\Delta}W$ along its singular locus $\widetilde{\Gamma}$ (see Remark \ref{singtilde}).
By restricting the morphisms in the upper part of   the diagram in  Corollary \ref{diagram cover}, we get the diagram
\[
\xymatrix{ \widetilde{\Gamma} \ar[d]_{a_{3}} & \widehat{\Gamma} \ar[l]_{a_{1}} \ar[d]_{a_{4}} \ar[r]^{a_{2}} & I \ar[d]_{a_{5}}\\  
\widetilde{\Omega} & \widehat{\Omega} \ar[l]_{a_{6}} \ar[r]^{a_{7}} & \overline{\Omega}.\\
}
\]
As $\widetilde{\Omega}$ is contained in the branch locus of $\widetilde{\epsilon}$, the morphism $a_{4}$ is an isomorphism and 
$\widetilde{\Gamma}$ and $\widetilde{\Omega}$ are isomorphic to $G$. By item (1) of Corollary \ref{lift}, the exceptional divisor  $\widehat{\Gamma}$ has a natural identification with
$\P( \mc U )\times_G \P (\mc U)$ and $a_{1}$ is the natural fibration over $G$. Analogously, by Remark \ref{R1}, 
the divisor  $\widehat{\Omega}$ is identified with $\P ( \Sym^2 \mc U)=\P(\widetilde{Z})$ and $a_{6}$ is the fibration over $G$.  

The restriction $a_{4}: \P (\mc U) \times_G \P (\mc U) \rightarrow\P ( \Sym^2 \mc U)$ of $\widehat{\varepsilon}_{1,v}$
is the natural $2:1$ morphism.

The birational morphism $a_{2}: \P (\mc U) \times_G \P (\mc U) \rightarrow I \subset \P(V)\times \P(V)$ is induced by composing with the natural morphism $\P (\mc U)\rightarrow \P(V)$.

The birational morphism  $a_{7}: \P(\widetilde{Z})\rightarrow Z$ is induced by $f$
and finally the finite $2:1$ morphism $a_{5}: I\rightarrow  \overline{\Omega}=\P(Z)$  is the map $\P(\varepsilon)$ obtained from $\varepsilon$
by projectivization (see the proof of Corollary \ref{overlineepsilon}).

\end{rem}

\section{The global covering} \label{global}

In this section we globalize the local double coverings of Lemma \ref{diagram cover}.\footnote{The varieties that we construct in this section depend on the abelian surface $A$ and the chosen $v$--generic polarization but, as in the previous sections, we omit this dependence to avoid cumbersome notation.}
Our starting point is the following result contained in \cite{rap_phd} and \cite{perrapimrn}.
Keeping the notation as above,  let $ \widetilde{\Sigma}_{v} \subset \Kvt$ be the exceptional
divisor of the blow up $\Kvt=Bl_\Sigma \Kv \to \Kv$.

\begin{thm} \cite[Theorem 3.3.1]{rap_phd}
The class of $\widetilde{\Sigma}_{v}$ in the Picard group $Pic(\Kvt)$ of  $\Kvt$ is divisible by two.
\end{thm}
\begin{proof}
The case of $\widetilde{K}_{(2,0,-2)}$ is dealt in  \cite[Theorem 3.3.1]{rap_phd}.
The general case follows from Theorem 3.1 and Remark 3.4 of \cite{perrapimrn}.
\end{proof}

As the Picard group of the  IHS manifold $\Kvt$ is torsion free,
there exists a unique normal projective variety $\widetilde{Y}_{v}$ equipped with a double cover   $\widetilde{\epsilon}_{v}:\widetilde{Y}_{v}\rightarrow \Kvt$
 branched over  $\widetilde{\Sigma}_{v}$.  This double cover allows us to construct the global analogue of the morphism
$\varepsilon$ of Lemma \ref{diagram cover}.

\begin{thm}\label{global2:1}
There exists a unique normal projective variety $Y_{v}$ equipped with a finite $2:1$ morphism 
$\varepsilon_{v}:Y_{v}\rightarrow \Kv$ whose branch locus is $\Sigma_{v}$. The ramified double cover induced by   
$\varepsilon_{v}$ on a small analytic neighborhood of a point of $\Omega_{v}$ is isomorphic to the ramified double cover induced by 
$\varepsilon: W \to Z$ on  a small analytic neighborhood of a point of $\Omega$ in $Z$.
\end{thm}
\begin{proof}
For any $p\in \Omega_{v}$ there exists a small analytic neighborhood $U_{p,v}$ of $p\in\Kv$ that is biholomorphic to the intersection of 
$Z$ with an open ball. Hence, for any $p\in\Omega_{v}$ there exists a proper complex analytic space $Y_{p,v}$ and a finite $2:1$ morphism  
$\varepsilon_{p,v}: Y_{p,v}\rightarrow U_{p,v}$ branched along $U_{p,v}\cap \Sigma_{v}$, which is obtained by restricting $\varepsilon$.

On the other hand, there exists an analytic manifold $Y^{o}_{v}$ and a finite $2:1$ morphism 
$\varepsilon^{o}_{v}:Y^{o}_{v} \rightarrow \Kv\setminus\Omega_{v}$
 branched along $\Sigma_{v}\setminus\Omega_{v}$. To see this, first of all notice that by restricting 
$\widetilde{\epsilon}_{v}:\widetilde{Y}_{v}\rightarrow \Kvt$ we get a double covering of $\Kvt \setminus \widetilde{\Omega}_{v}$
branched along 
$\widetilde{\Sigma}_{v}\setminus \widetilde{\Omega}_{v}$. 
Since $\Kv$ has an $A_1$ singularity along $\Sigma_{v}$,
the exceptional divisor $\widetilde{\Sigma}_{v}\setminus \widetilde{\Omega}$ is a $\mathbb{P}^{1}$--bundle whose normal bundle has degree $-2$
on the fibers. It follows that $\widetilde{\varepsilon}_{v}^{-1}(\widetilde{\Sigma}_{v}\setminus \widetilde{\Omega})$
is a $\mathbb{P}^{1}$--bundle whose normal bundle has degree $-1$
on the fibers. By Nakano's Theorem (see \cite{nak} and \cite{fujnak} ),  $\widetilde{\epsilon}_{v}^{-1}(\Kvt \setminus\widetilde{\Omega})$
is the blow up of a complex manifold $Y^{o}_{v}$ along a submanifold isomorphic to $\Sigma_{v}\setminus\Omega_{v}$.
Moreover, since $f_{v}\circ \widetilde{\epsilon}_{v}: \widetilde{Y}_{v}\rightarrow \Kv$ is constant on the fibers 
of the $\mathbb{P}^{1}$ bundle $\widetilde{\epsilon}_{v}^{-1}(\widetilde{\Sigma}_{v}\setminus \widetilde{\Omega})$, it induces  the desired  finite $2:1$ morphism
$\varepsilon^{o}_{v}:Y^{o}_{v}\rightarrow \Kv\setminus\Omega_{v}$.

To yield the existence of $\varepsilon_{v}:Y_{v}\rightarrow \Kv$, it will suffice to prove that 
$\varepsilon_{p,v}$ and $\varepsilon^{o}_{v}$ induce isomorphic double covers on
$U_{p}\setminus \{p\}$ so that they can be glued to get $\varepsilon_{v}$.
Recall from Remark  \ref{ridimo} that the fundamental group of $Z\setminus\Sigma$ is $\mathbb{Z}/2\mathbb{Z}$ . Since the same holds for $U_{p,v}\setminus \Sigma$, the \'etale double covers induced by $\varepsilon_{p,v}$ and $\varepsilon^{o}_{v}$ on 
$U_{p,v}\setminus \Sigma$ are isomorphic. The closure of the graph of this isomorphism in the fiber product 
$\varepsilon^{o-1}_{v}(U_{p,v}\setminus \{p\})\times_{U_{p,v}\setminus \{p\}} \varepsilon_{p,v}^{-1}(U_{p}\setminus \{p\})$
is finite and bimeromorphic on the manifolds $\varepsilon^{o-1}_{v}(U_{p,v}\setminus \{p\})$ and $\varepsilon_{p,v}^{-1}(U_{p}\setminus \{p\})$,
hence it is the graph of an isomorphism of double covers $\Kv$.

The glued complex analytic space   $Y_{v}$ is also projective as a consequence of GAGA's principles \cite[Cor 4.6]{SGA1}, since it has a finite proper map to a projective variety. Finally $Y_{v}$ is normal since $W$ is normal and since the normality of a complex variety may be checked on the associated complex analytic space
(\cite[Prop. 2.1]{SGA1}).

To prove uniqueness of $\varepsilon_{v}$, let $\varepsilon_{v}':Y_{v}'\rightarrow K_{v}$ be a finite $2:1$  morphism branched  over $\Sigma_{v}$ such that
$Y_{v}'$ is normal.  
In this case $Y_{v}\setminus \varepsilon_{v}^{-1}(\Sigma_{v})$ and $Y_{v}'\setminus \varepsilon_{v}'^{-1}(\Sigma_{v})$ are algebraic proper \'etale double covers 
of $\Kv\setminus \Sigma_{v}= \Kvt\setminus \widetilde{\Sigma}_{v}$. Any such cover is determined by a $2$ torsion point in the Picard group 
$Pic(\Kvt\setminus \widetilde{\Sigma}_{v})$ and a nowhere vanishing section (unique up to scalars) of the trivial line bundle. As the  $\widetilde{\Sigma}_{v}$ is irreducible  and its class is divisible by $2$ in the free
group $Pic(\Kvt)$, there exists a unique non trivial $2$ torsion point in $Pic(\Kvt\setminus \widetilde{\Sigma}_{v})$. Moreover, as  $\Kv$ is normal and $\Sigma_{v}$
has codimension $2$ in $\Kv$, a  regular function on $\Kv\setminus \Sigma_{v}$ extends to the projective variety $\Kv$ and therefore it is constant.
It follows that      
$Y_{v}\setminus \varepsilon_{v}^{-1}(\Sigma_{v})$ and $Y_{v}'\setminus \varepsilon_{v}'^{-1}(\Sigma_{v})$ are isomorphic \'etale double covers of 
 $\Kv\setminus \Sigma_{v}$. 

Repeating the  argument in the final part of the proof of the existence, the closure of the graph of this isomorphism in the
fiber product $Y_{v}\times_{\Kv}Y_{v}'$ is finite and birational over the normal varieties $Y_{v}$ and $Y_{v}'$, 
hence it is the graph of an isomorphism of double covers.
The local characterization of $\varepsilon$ near points of $\Omega_{v}$ holds by construction.
\end{proof}

Theorem \ref{global2:1} allows to prove a straightforward global version of Lemma \ref{diagram cover}. 
Let $\Delta_{v}\subset Y_{v}$ be the ramification locus (with the reduced induced structure) of $\varepsilon_{v}$
and let $\Gamma_{v}$ be the singular locus (consisting of 256 points) of $Y_{v}$.
Denote by $\widetilde{\Gamma}_{v}$  the inverse image with reduced structure of $\Gamma_{v}$ in  
$Bl_{\Delta_{v}}Y_{v}$ and denote by $\overline{\Delta}_{v}$ the strict transform of $\Delta_{v}$ in $Bl_{\Gamma_{v}}Y_{v}$.  
\begin{cor} \label{global diagram cover}
\begin{enumerate}
\item{}The projective varieties  $Bl_{\widetilde{\Gamma}_{v}}Bl_{\Delta_{v}}Y_{v}$ and  $Bl_{\overline{\Delta}_{v}}Bl_{\Gamma_{v}}Y_{v}$ are smooth and isomorphic over $Y_{v}$.

\item{}There exist finite $2:1$ morphisms $\widetilde{\varepsilon}_{v}:Bl_{\Delta_{v}}Y_{v}\rightarrow \Kvt$, 
$\overline{\varepsilon}_{v}:Bl_{\Gamma_{v}}Y_{v}\rightarrow Bl_{\Omega_{v}}K_{v}$,
$\widehat{\varepsilon}_{1,v}: Bl_{\widetilde{\Gamma}_{v}}Bl_{\Delta_{v}}Y_{v}\rightarrow Bl_{\widetilde{\Omega}_{v}}Bl_{\Sigma_{v}} \Kv$
and $\widehat{\varepsilon}_{2,v}: Bl_{\overline{\Delta}_{v}}Bl_{\Gamma_{v}}Y_{v} \rightarrow Bl_{\overline{\Sigma}_{v}}Bl_{\Omega_{v}} \Kv$, 
lifting $\varepsilon_{v}$. Hence, there exists  a commutative diagram 
\[
\xymatrix{
&   Bl_{\widetilde{\Gamma}_{v}}Bl_{\Delta_{v}}Y_{v}=Bl_{\overline{\Delta}_{v}}Bl_{\Gamma_{v}}Y_{v}
\ar[dr] \ar[d]_{\widehat{\varepsilon}_{1,v}=\widehat{\varepsilon}_{2,v}}  \ar[dl] & \\
Bl_{\Delta_{v}}Y_{v}\ar[d]_{\widetilde{\varepsilon}_{v}} \ar[dr] &  
Bl_{\widetilde{\Omega}_{v}}Bl_{\Sigma_{v}} \Kv = Bl_{\overline{\Sigma}_{v}}Bl_{\Omega_{v}}\Kv \ar[dl] \ar[dr] \Kv & Bl_{\Gamma_{v}}Y_{v} 
\ar[d]_{\overline{\varepsilon}_{v}}
\ar[dl] 
\\
\Kvt =Bl_{\Sigma_{v}}\Kv \ar[dr]& Y_{v}\ar[d]_{\varepsilon_{v}} & Bl_{\Omega_{v}} \Kv\ar[dl]\\
& \Kv & \\
}
\]
where the diagonal arrows are blow ups.
\end{enumerate}
\end{cor}

\begin{proof}
(1) Over the inverse images of the smooth locus of $Y_{v}$, the existence of the isomorphism is trivial, whereas over the inverse images of  small euclidean neighborhoods of the singular points of $Y_{v}$ it follows from  item (2) of Lemma \ref{diagram cover}. Since the global blow up is obtained by gluing local blow ups, 
$Bl_{\widetilde{\Gamma}_{v}}Bl_{\Delta_{v}}Y_{v}$ and $Bl_{\overline{\Delta}_{v}}Bl_{\Gamma_{v}}Y_{v}$ are isomorphic over $Y_{v}$. 
The smoothness of  $Bl_{\overline{\Delta}_{v}}Bl_{\Gamma_{v}}Y_{v}$ and 
$Bl_{\overline{\Delta}_{v}}Bl_{\Gamma_{v}}Y_{v}$ follows from the smoothness of 
$Bl_{\overline{\Delta}}Bl_{\Gamma}W$ and 
$Bl_{\overline{\Delta}}Bl_{\Gamma}W$, which was proven in Corollary \ref{diagram cover}. Hence 
item (1) holds.

(2) The existence of the liftings of the double cover $\varepsilon_{v}$ over the inverse images of $K_{v}\setminus \Omega_{v}$ 
 is clear. 
Over the inverse images of a small euclidean neighborhood  in $K_{v}$ of a point of $\Omega_{v}$, the existence of the lift follows from (a) of Theorem \ref{lehn sorger},  Corollary \ref{overlineepsilon}, Corollary \ref{lift}, and from item (1) of Lemma 
\ref{diagram cover}.
Since, the lift of a morphism to bimeromorphic varieties is unique, whenever it exists, it is possible to glue the local liftings and obtain the desired global morphism.
\end{proof}

\begin{rem}\label{rem:ultimo}
In Corollary \ref{global diagram cover} we have showed that $Bl_{\Delta_{v}}Y_{v}$ is a double cover of $\Kvt$ and  that it branched over 
$\widetilde{\Sigma}_{v}$. Moreover, by Corollary \ref{lift}, the projective variety  $Bl_{\Delta_{v}}Y_{v}$ is normal.
Since the Picard group of the  IHS manifold $\Kvt$ is torsion free,
there exists a unique  such a double cover. It follows that $Bl_{\Delta_{v}}Y_{v}=\widetilde{Y}_{v}$ and 
$\widetilde{\varepsilon}_v= \widetilde{\epsilon}_v$. 
\end{rem}
In order to describe  the ramification loci of these double coverings, we need to introduce some further notation.
In the following corollary we  denote by $\widetilde{\Delta}_{v}\subset Bl_{\Delta_{v}}Y_{v}$ the exceptional divisor and by   
$\widehat{\Delta}_{v}\subset Bl_{\widetilde{\Gamma}_{v}}Bl_{\Delta_{v}}Y_{v}=Bl_{\overline{\Delta}_{v}}Bl_{\Gamma_{v}}Y_{v}$
the strict transform of $\widetilde{\Delta}_{v}$ or, equivalently, the exceptional divisor of the blow up of $Bl_{\Gamma_{v}}Y_{v}$ along 
$\overline{\Delta}_{v}$.

\begin{cor} \label{cor:rambra}
\begin{enumerate}
\item{}The branch loci of $\widetilde{\varepsilon}_{v}$, $\overline{\varepsilon}_{v}$, and of $\widehat{\varepsilon}_{1,v}(=\widehat{\varepsilon}_{2,v})$
are $\widetilde{\Sigma}_{v}$, $\overline{\Sigma}_{v}$, and $\widehat{\Sigma}_{v}$, respectively.
\item{}The ramification loci of $\widetilde{\varepsilon}_{v}$, $\overline{\varepsilon}_{v}$ and $\widehat{\varepsilon}_{1,v}(=\widehat{\varepsilon}_{2,v})$
are $\widetilde{\Delta}_{v}$, $\overline{\Delta}_{v}$, and $\widehat{\Delta}_{v}$, respectively.
\end{enumerate}
\end{cor}
\begin{proof} The statements on the branch loci are determined by the analogous statement proved for the local case. Specifically, (1) follows from Corollary \ref{lift}, Corollary \ref{overlineepsilon}
and item (1) of Corollary \ref{diagram cover}.

Since the ramification locus of $\varepsilon_{v}$ is $\Delta_{v}$ and its branch locus is $\Sigma_{v}$, (2) follows from (1) and from the commutativity of the diagram in item (2) of  
Corollary \ref{global diagram cover}.
\end{proof}
In the final part of this section on the global geometry of the double covers induced by $\varepsilon_{v}$, we compare their ramification and their branch loci
and discuss the associated involutions. 

\begin{rem} \label{rmk:esplsigma}
Since $\Sigma_{v}\simeq A\times A^{\vee}/\pm1$, $\widetilde{\Sigma}_{v}$ has an $A_{1}$ singularity along its singular locus (see Remark \ref{R2}). Moreover,  
$\overline{\Sigma}_{v}$ and $\widehat{\Sigma}_{v}$ are smooth (see Proposition \ref{fundamental diagram}) and
the branch loci of $\varepsilon_{v}$, $\widetilde{\varepsilon}_{v}$, $\overline{\varepsilon}_{v}$, and $\widehat{\varepsilon}_{1,v}(=\widehat{\varepsilon}_{2,v})$ are normal. Hence these double covers induce  isomorphisms $\Delta_{v}\simeq\Sigma_{v}$, $\widetilde{\Delta}_{v}\simeq\widetilde{\Sigma}_{v}$, 
$\overline{\Delta}_{v}\simeq\overline{\Sigma}_{v}$, and $\widehat{\Delta}_{v}\simeq \widehat{\Sigma}_{v}$.
\end{rem}

\begin{rem}\label{involuzioni}
Remark \ref{lisci su} implies that $Bl_{\Gamma_{v}}Y_{v}$ and 
$Bl_{\widetilde{\Gamma}_{v}}Bl_{\Delta_{v}}Y_{v}=Bl_{\overline{\Delta}_{v}}Bl_{\Gamma_{v}}Y_{v}$
are smooth, Corollary \ref{lift} implies that $Bl_{\Delta_{v}}Y_{v}$ is normal and, by Theorem \ref{global2:1}, $Y_{v}$ is normal too.
Hence the finite $2:1$ morphisms $\varepsilon_{v}$, $\widetilde{\varepsilon}_{v}$, $\overline{\varepsilon}_{v}$ and $\widehat{\varepsilon}_{1,v}(=\widehat{\varepsilon}_{2,v})$
induce regular involutions $\tau_{v}$, $\widetilde{\tau}_{v}$, $\overline{\tau}_{v}$, $\widehat{\tau}_{1,v}$ and $\widehat{\tau}_{2,v}$ on $Y_{v}$, $Bl_{\Delta_{v}}Y_{v}$,
$Bl_{\Gamma_{v}}Y_{v}$, $Bl_{\widetilde{\Gamma}_{v}}Bl_{\Delta_{v}}Y_{v}$ and $Bl_{\overline{\Delta}_{v}}Bl_{\Gamma_{v}}Y_{v}$ respectively. Recall that $K_{v}$ is normal, and hence so is $Bl_{\Omega_{v}}K_{v}$ by Proposition \ref{local fundamental diagram}. As for  $\Kvt$ and $Bl_{\widetilde{\Omega}_{v}}Bl_{\Sigma_{v}} \Kv = Bl_{\overline{\Sigma}_{v}}Bl_{\Omega_{v}}\Kv$, they are both smooth. It follows that the morphisms $\varepsilon_{v}$, $\widetilde{\varepsilon}_{v}$, $\overline{\varepsilon}_{v}$ and $\widehat{\varepsilon}_{1,v}(=\widehat{\varepsilon}_{2,v})$ can be identified with the quotient maps of the respective involutions $\tau_{v}$, $\widetilde{\tau}_{v}$, $\overline{\tau}_{v}$, $\widehat{\tau}_{1,v}$ and $\widehat{\tau}_{2,v}$.  
\end{rem}

\section{The birational geometry of $Y_{v}$} \label{birgeom}

In this section we describe the global geometry of $Y_{v}$, we show that
it is birational to an IHS manifold
of $K3^{[n]}$ type, and we  describe explicitly the birational map.

In the first part of the section, we consider the special case where  $A$ is a principally polarized abelian surface, whose N\'eron--Severi group  ois generated by the principal symmetric 
polarization $\Theta$.
As is well known, the linear system $|2\Theta|$ defines a morphism
$g_{|2\Theta|}:A\rightarrow |2\Theta|^{\vee}\simeq\mathbb{P}^{3}$
whose image 
is the singular Kummer surface $Kum_{s}$ of $A$, a nodal quartic surface
isomorphic to the quotient $A/\pm 1$. The smooth Kummer surface
$S$ of  $A$ is the blow up of $Kum_{s}$ along the singular locus $A[2]$. 

We are
going to show that in this case
$Y_{(0,2\Theta,2 )}$ is birational to the Hilbert scheme $S^{[3]}$. The
following remark collects some known results that we need in the proof.

\begin{rem}\label{riassunto}
\begin{enumerate}
\item{}The locus of $|2\Theta|$ parametrizing singular curves consists of $17$
irreducible divisors:
the divisor $R$ parametrizing reducible curves and, for any $2$ torsion
point $\alpha\in A[2]$,
the divisor  $N_{\alpha}$ parametrizing curves passing through $\alpha$.
The divisor $R$ is isomorphic to $Kum_{s}$ and a general point of $R$ corresponds to a
curve of the form $\Theta_{x}\cup \Theta_{-x}$, where $\Theta_{x}$ and $ \Theta_{-x}$ meet transversally outside of $A[2]$.
For every $ \alpha \in A[2]$, the divisor  $N_{\alpha}$ is isomorphic to $\mathbb{P}^{2}$ and
parametrizes curves whose images in $Kum_{s}$
are plane sections through the singular point $g_{|2\Theta|}(\alpha)$. The
general point of $N_{\alpha}$ corresponds to a curve $C$ that is a double cover of a quartic plane curve
with precisely one node; this double covers ramifies over the node and therefore $C$ has a node in $\alpha$ and no other singularity.

\item{}For the natural choices in the definition of the map $$\mathbf{a}_{(0,2\Theta,2 )}:M_{(0,2\Theta,2 )}(A,\Theta)\rightarrow A\times A^{\vee},$$ 
(see Introduction formula (\ref{Kv})), the subvariety  $K_{(0,2\Theta,2 )}:=\mathbf{a}_{(0,2\Theta,2 )}^{-1}(0,0)\subset M_{(0,2\Theta,2 )}(A,\Theta)$ parametrizes sheaves  whose determinant is equal to $\mathcal{O}(2\Theta)$  and whose second Chern class sums up to $0\in A$.
Since  $M_{(0,2\Theta,2 )}(A,\Theta)$ parametrizes pure dimension $1$ sheaves, there
exists a regular morphism $t: \Kv\rightarrow
|2\Theta|\simeq\mathbb{P}^{3}$, called the support morphism, which to every polystable sheaf associates
its Fitting subscheme (see \cite{lep}).
The morphism $t$ is surjective and since $K_{(0,2\Theta,2 )}$ has a resolution that is a IHS
manifold, all its fibers are $3$-dimensional.


\item{}Since the polarization $\Theta$ is symmetric, $-1^{*}$ induces an involution
on the moduli space $M_{(0,2\Theta,2 )}(A,\Theta)$ whose fixed locus contains
the variety  $K_{(0,2\Theta,2)}$. Indeed, any smooth curve $C\in |2\Theta |$
 is an \'etale double cover of  its image $g_{|2\Theta|}(C)$.
The pull back to $C$ of any degree--$3$ line bundle on  $g_{|2\Theta|}(C)$, is a stable sheaf of
$K_{(0,2\Theta,2)}$ which is $-1^*$--invariant. Moreover,  the pullback of two line bundles
on $g_{|2\Theta|}(C)$ are isomorphic if and
only if the  two line bundles
 differ by $2$ torsion line
bundle defining the \'etale double cover $C\rightarrow g_{|2\Theta|}(C)$.
Hence,  there exists a six dimensional algebraic subset of
$K_{(0,2\Theta,2)}$ that is fixed by the involution and hence,
by closure of the fixed locus,
the whole $K_{(0,2\Theta,2)}$ is fixed.

\item{}
If $C\in|2\Theta|$ is smooth or general in $R$ or $N_{\alpha}$, the general point of $t^{-1}(C)$
represents a sheaf that is locally free on its support. 
This holds because, for any nodal curve $C$, any torsion free sheaf on $C$ that is not locally free, 
is the limit of locally free sheaves on $C$ varying
in a family parametrized by $\mathbb{P}^{1}$.  Since any $\mathbb{P}^1$ has to be contracted by $\mathbf{a}_{(0,2\Theta,2 )}$, it follows that it has to be contained in $K_{(0,2\Theta,2)}$ and the claim follows.

\item{ }The inverse image on $K_{(0,2\Theta,2)}$ of an irreducible  surface contained in $|2\Theta|\simeq\mathbb{P}^{3}$
is irreducible.
Since $t$ is equidimensional, it suffices to prove that $t^{-1}(C)$ is irreducible for any curve $C$  that is
 smooth or general in $R$ or in $N_{\alpha}$. If $C$ is such a curve,  the locus $t^{-1}(C)^{lf}$ parametrizing sheaves in 
$t^{-1}(C)$ that are locally free on their support is dense in $t^{-1}(C)$ and, moreover, $-1$ has at most $1$ fixed point on $C$. It follows that, any $F\in t^{-1}(C)^{lf}$, the $-1$ action on $F$ 
can be linearized  in such a way that the action is trivial on the fiber over the fixed point. By Kempf descend Lemma 
(see Theorem 4.2.15 of  
\cite{HL}), this means that $F$ is the pull back of a line bundle on the irreducible nodal curve $g_{|2\Theta|}(C)$. 
Since the generalized Jacobian of an irreducible nodal curve is irreducible, $t^{-1}(C)^{lf}$ and its closure $t^{-1}(C)$ are also irreducible.
\end{enumerate}
\end{rem}

\begin{lem}\label{MZ}
Let  $A$ be a principally polarized abelian surface, with $NS(A)=\mathbb{Z}\Theta$.
Then $\widetilde{Y}_{(0,2\Theta,2)}$ is birational to the Hilbert scheme $S^{[3]}$. 
\end{lem}
 
\begin{proof}
 Let $D$ be the pull back on $S$ of a plane section of $Kum_s$.
We are going to show that $\widetilde{Y}_{(0,\Theta,2)}$ is birational to the smooth  projective moduli space $M_{(0,D,1)}$ parametrizing  sheaves on $S$
with Mukai vector $(0,D,1)$ and which are stable with respect to a fixed $(0,D,1)$--generic polarization. The moduli space  
$M_{(0,D,1)}$ is well known to be birational to  $S^{[3]}$ (see Proposition 1.3 of \cite{bea}).

By construction, there exists an isomorphism between linear systems $\psi:|D|\rightarrow |2\Theta|$. 
Moreover, any sheaf $F\in M_{(0,D,1)}$, 
whose support is a smooth curve, may be seen as a sheaf on $Kum_{s}$  and its pull back to $A$ is a stable  sheaf of $K_{(0,2\Theta,2)}$. It follows that
there exists a commutative diagram   
\begin{equation}
\label{DZ}
\xymatrix{
&  \widetilde{K}_{(0,2\Theta,2)}\ar[d]^{f_{(0,2\Theta,2)}}\\
  M_{(0,D,1)}\ar[d]_{s} \ar@{-->}[ur]^{\phi}  & K_{(0,2\Theta,2)} \ar[d]^{t}\\
  |D| \ar[r]_{\psi} & |2\Theta| ,\\
}
\end{equation}
where $s$ and $t$ are the two support morphisms. 
If $C\in|2\Theta|$ is a smooth curve, it is a connected \'etale double cover of the smooth curve $g_{|2\Theta|}(C)$ and since $g_{|2\Theta|}(C)\cap A[2]=\emptyset$, it
can be considered as a curve in $|D|$.
The restriction of $\phi$ on  
$s^{-1} (g_{|2\Theta|}(C))\simeq Pic^{3}(g_{|2\Theta|}(C))$  is therefore well defined and  gives  an  \'etale  double cover of  $(t\circ f_{(0,2\Theta,2)})^{-1} (C))
\simeq t^{-1} (C)\subset Pic^6(C)$ (see (3) of Remark \ref{riassunto}).
This shows that $\phi$ is a rational map of degree $2$.

In order to compare $\widetilde{Y}_{(0,2\Theta,2)}$ and $M_{(0,D,1)}$, we need to determine the branch divisor $B$ of $\phi$, i.e. the divisor on $\Kvt$
where a resolution of the indeterminacy of $\phi$ is not \'etale.
We have already seen that $B$ has to parametrize sheaves supported on singular curves.

Let $U\subset M_{(0,D,1)}$ be the biggest open subset where $\phi$ extends to a regular morphism. As   $M_{(0,D,1)}$ and $K_{(0,2\Theta,2)}$ have trivial 
canonical bundle, the differential of $\phi$ is an isomorphism at any point of $U$.
As a consequence, $\phi$ does not contract any positive dimensional subvariety  and $\phi(U)\subset \widetilde{K}_{(0,2\Theta,2)}$ is an open subset.

We claim that the open subset  $\phi(U)$ intersects any divisor of $\widetilde{K}_{(0,2\Theta,2)}$ with the possible exception of 
$\widetilde{\Sigma}_{(0,2\Theta,2)}$.
Since we have already shown that $(t\circ f_{(0,2\Theta,2)})^{-1} (C))\in \phi(U)$ if $C$ is smooth,  it remains to check this statement for divisors contained in
$(t\circ f_{(0,2\Theta,2)})^{-1} (R))$ and $(t\circ f_{(0,2\Theta,2)})^{-1} (N_{\alpha}))$. As $\Sigma_{(0,2\Theta,2)}\subset t^{-1}(R)$, 
by 5) of Remark \ref{riassunto},
the divisor  $(t\circ f_{(0,2\Theta,2)})^{-1} (R)$ is the union of $\widetilde{\Sigma}_{(0,2\Theta,2)}$ 
and the strict transform of $t^{-1}(R)$. Finally, $(t\circ f_{(0,2\Theta,2)})^{-1}N_{\alpha}$ is irreducible.

The general point $F$  of $t^{-1}(R)$ represents a line bundle supported on the general curve $C$ of $R$. As $C$ does not intersect $A[2]$, its image 
$g_{|2\Theta|}(C)\simeq C/\pm 1$ may be seen as a curve in $D$ and, as in the smooth case, $F$ descends to a line bundle on  $g_{|2\Theta|}(C)$. Hence $\phi(U)$ intersects
the strict transform of $t^{-1}(R)$.

Finally, by commutativity of diagram \eqref{DZ}, $\phi$ sends an open subset of $(\psi\circ s)^{-1}(N_{\alpha})$ to the irreducible divisor   
$(t\circ f_{(0,2\Theta,2)})^{-1}N_{\alpha}$. Since $(\psi\circ s)^{-1}(N_{\alpha})$ cannot be contracted, $\phi(U)$ also intersects 
$(t\circ f_{(0,2\Theta,2)})^{-1}N_{\alpha}$. This completes the proof of our claim.
  
Let $r:N\rightarrow \Kvt$ be a resolution of  the indeterminacy of $\phi$, hence $N$ ia smooth projective variety such that there exists
a commutative diagram 
\[\xymatrix{
 N\ar[d]^{b}  \ar[r]^{r} &  \widetilde{K}_{(0,2\Theta,2)}\\
  M_{(0,D,1)}\ar@{-->}[ur]^{\phi}\\
}
\]
where $b$ is birational and induces an isomorphism between $b^{-1}(U)$ and $U$.
Let $U'\subset \widetilde{K}_{(0,2\Theta,2)}\setminus \widetilde{\Omega}_{v}$ be the open subset where the fibers of $\xi$ are $0$ dimensional. 
Notice that, since $N$ and $\widetilde{K}_{(0,2\Theta,2)}$ are smooth, $r$ is flat over $U'$.
Therefore $r^{-1}(U')$ is a flat ramified double cover of $U'$ and, since 
$\widetilde{K}_{(0,2\Theta,2)}\setminus U'$ has codimension at least $2$ and $Pic(\widetilde{K}_{(0,2\Theta,2)})$ is torsion free, 
this double cover is determined by its branch divisor $B$, i.e. by the locus where fibers of 
$r$ are length $2$ non reduced subschemes.

We already  know that if $p\in U'\cap \phi(U)$ the fiber $r^{-1}(p)$ has a at least one component consisting of a reduced point of $U$. 
Hence $B$ is a divisor contained in $U'\setminus \phi(U)$: therefore, by our claim, either $B$ is empty or $B=U'\cap \widetilde{\Sigma}_{(0,2\Theta,2)}$.
The first case is impossible becouse $U'$ is simply connected and $N$ is irreducible. In the second case $r^{-1}(U')$ is the unique double cover of $U'$
ramified over $U'\cap \widetilde{\Sigma}_{(0,2\Theta,2)}$, hence $r^{-1}(U')$ is isomorphic to $\varepsilon_{v}^{-1}(U')\subset \widetilde{Y}_{(0,2\Theta,2)}$.
\end{proof}

The following Proposition  generalizes Lemma \ref{MZ}, by showing that $Y_{v}$
is always birational to an IHS manifold, and  describes a resolution of the indeterminacy of the birational map.

Recall that the exceptional divisor $\overline{\Gamma}_{v}$ of $Bl_{\Gamma_{v}}Y_{v}$ consists of the disjoint union of $256$ copies $I_{i,v}$ 
of the incidence variety $I\subset\mathbb{P}(V)\times\mathbb{P}(V)$,
each of which has two natural $\mathbb{P}^{2}$ fibrations given by the projections onto   $\mathbb{P}(V)$.
For any $i$, we let $p_{i}:I_{i,v}\rightarrow \mathbb{P}(V)$ be one of the $2$ projections. 
Since $Y_{v}$ is locally analytically isomorphic to the cone $W$, the  normal bundle of $I_{i,v}$ in $Bl_{\Gamma_{v}}Y_{v}$ has degree $-1$ on the fibers of $p_{i}$.

Therefore, by applying Nakano's contraction Theorem  (\cite{nak}),  there exists a complex manifold $\underline{Y}_{v}$ and a morphism
of complex manifolds $h_{v}:Bl_{\Gamma_{v}}Y_{v}\rightarrow \underline{Y}_{v}$
whose  exceptional locus  is $\overline{\Gamma}_{v}$ and is  
such that the image  $J_{i,v}:=h_{v}(I_{i,v})$ of  any component $\overline{\Gamma}_{v}$ is isomorphic to $\mathbb{P}^{3}$. Moreover, the restriction of 
$h_{v}$ on $I_{i,v}$ equals $p_{i}$ and $h_{v}$ realizes $Bl_{\Gamma_{v}}Y_{v}$ as the blow up of $\underline{Y}_{v}$ along the disjoint union 
$J:=h_{v}(\overline{\Gamma}_{v})$ of the $J_{i,v}$'s.

\begin{prop}\label{prop:birh}
Keeping the notation as above, the complex manifold  $\underline{Y}_{v}$ is a projective IHS manifold that is deformation 
equivalent to the Hilbert scheme parametrizing 
$0$-dimensional subschemes of length $3$ on a $K3$ surface.

\end{prop}   

\begin{proof}
Notice that the ramification locus of $\varepsilon_{v}:Y_{v}\rightarrow \Kv$ has codimension $2$. It follows that the canonical divisor of
$Y_{v}$ is trivial and the canonical divisor of $Bl_{\Gamma_{v}}Y_{v}$ is supported on  $\overline{\Gamma}_{v}$.
As the normal bundle of $I_{i,v}$ in $Bl_{\Gamma_{v}}Y_{v}$ has degree $-1$ on the fibers of both the $\mathbb{P}^{2}$ fibrations  of $I_{i,v}$ , by adjunction, the canonical bundle of 
the smooth variety $Bl_{\Gamma_{v}}Y_{v}$ is $2\sum_{i=1}^{256} I_{i,v}$.

Let $r_{i}$ be a  line contained in a fiber of $p_{i}$ and let  $l_{i}$ be  a line contained in a fiber 
of the other  $\mathbb{P}^{2}$ fibration of $I_{i,v}$. A priori, it is not clear whether $r_{i}$ and $l_{i}$ are numerically equivalent. Nevertheless, since $Y_{v}$ is projective and $r_{i}$ and $l_{i}$ generate the cone of effective curves on $I_{i,v}$, the set $\{r_{i}\}$ represents $256$
 $K_{Bl_{\Gamma_{v}}Y_{v}}$-negative extremal rays of the Mori cone of $Bl_{\Gamma_{v}}Y_{v}$.
If $r_i$ and $l_{i}$ are equivalent, the contraction of $r_i$ contracts $I_{i,v}$ to a point admitting a Zariski neighborhood isomorphic to a 
Zariski neighborhood of the $i$-th singular point of $\Gamma_{v}$ in the normal variety $Y_{v}$. 
If $r_i$ and $l_{i}$ are independent, 
the contraction of $r_{i}$ can be identified with the Nakano contraction
restricting to $p_{i}$ on $I_{i,v}$.

In any case, the contraction of $r_{i}$ is divisorial and, by Corollary 3.18 of \cite{koll_mori},
it produces only $\mathbb{Q}$--factorial singularities.
Hence, after $256$ extremal contractions  we terminate with a $\mathbb{Q}$--factorial variety  with trivial canonical divisor 
and with terminal singularities. If $v=(0,2\Theta,2)$, by Lemma \ref{MZ}, the variety $M_{(0,2\Theta,2)}$ is a minimal model of the IHS manifold    
$S^{[3]}$ and by a theorem due to Greb, C. Lehn, and Rollenske (see Proposition 6.4 of \cite{gls}) it is an  IHS manifold. 
In particular $r_i$ and $l_{i}$ are always numerically independent and this IHS manifold is isomorphic to the Nakano contraction $\underline{Y}_{v}$.

To deal with the genaral case, recall from \cite[Theorem 1.6]{per_rap} (and its proof)  that   the singular variety $K_{v}$ 
can be deformed to $K_{(0,2\Theta,2)}$ using only isomorphisms induced by Fourier-Mukai transform and  locally trivial deformations induced 
by deformation of the underlying abelian surface (see Proposition 2.16 of \cite{per_rap}). 

Extending the construction of Theorem \ref{global2:1} to the case of a locally trivial deformation,
it is also possible to deform $Y_{v}$ to $Y_{(0,2\Theta,2)}$ by a locally trivial deformation.
By blowing up the subvariety consisting of singular points of all fibers in the total space of the deformation, we get
that $Bl_{\Gamma_{v}}Y_{v}$  can be deformed to $Bl_{\Gamma_{(0,2\Theta,2)}}Y_{(0,2\Theta,2)}$.
Up to an \'etale base change on the base of the deformation, 
we may also assume that the exceptional divisor consists of $256$ connected components $\mathcal{I}_{i}$, 
each of which has two fibration and  one of them restricts to $p_{i}$ on $I_{i,v}$.
Applying again Nakano's Theorem,  the $\mathcal{I}_{i}$'s may be contracted respecting the chosen fibration.

As a consequence, the complex manifold $\underline{Y}_{v}$ obtained from   $Bl_{\Gamma_{v}}Y_{v}$ by contracting 
the $r_{i}$'s  is deformation equivalent (via smooth deformations) to an IHS manifold $\underline{Y}_{(0,2\Theta,2)}$ that is birational to 
$S^{[3]}$.
It remains to show that  $\underline{Y}_{v}$ is projective. As in the case $v=(0,2\Theta,2)$, it suffices to show that 
$r_{i}$ and $l_{i}$  are numerically independent. This is true because parallel transport  preserves numerical independence  and the analogous statement
has been shown to hold on $Bl_{\Gamma_{(0,2\Theta,2)}}Y_{(0,2\Theta,2)}$.
\end{proof}

By construction, $\underline{Y}_{v}$ has a regular birational morphism to $Y_{v}$ contracting  $J$ to $\Gamma_{v}$.
In the following remark we show that the involution $\tau_{v}$ on $Y_{v}$ cannot be lifted to a regular involution on 
$\underline{Y}_{v}$.

\begin{rem}\label{rmk:ratinv}
Since the involution $\overline{\tau}_{v}:Bl_{\Gamma_{v}}Y_{v}\rightarrow Bl_{\Gamma_{v}}Y_{v}$ sends $\Gamma_{v}$ to itself, 
it descends to a rational involution
$\tau_{T}:\underline{Y}_{v}\dashrightarrow \underline{Y}_{v}$  restricting to a regular involution  on  the complement $\underline{Y}_{v}\setminus J_{v}$ of the 
union of the projective spaces $J_{i,v}$
in  $\underline{Y}_{v}$. Since, by definiton of $\tau$, the involution  $\overline{\tau}_{v}$ exchanges the two $\mathbb{P}^{2}$ fibrations on  $I_{i,v}$, 
the indeterminacy locus of $\tau_{T}$ is $J_{v}$. 
Finally, since $Bl_{\Gamma_{v}}Y_{v}\simeq Bl_{J_{v}}\underline{Y}_{v}$, the rational involution $\tau_{T}$ may be described as the composition
of a Mukai flop along $J_{v}$ and an isomorphism outside of this locus. 

\end{rem}

\section{The Hodge numbers}\label{sec:hodge}

Collecting the results of the previous sections, we finally  present  the new construction of  
$\Kvt$  that allows us to calculate Betti and Hodge numbers of $\Kvt$. 

To simplify notation, let us set
\[
\begin{aligned}
&\widehat{Y}_{v}:=Bl_{\widetilde{\Gamma}_{v}}Bl_{\Delta_{v}}Y_{v}=Bl_{\overline{\Delta}_{v}}Bl_{\Gamma_{v}}Y_{v}, \\
&\widehat{K}_{v}:=Bl_{\widetilde{\Sigma}_{v}}Bl_{\Omega_{v}}K_{v}=Bl_{\overline{\Delta}_{v}}Bl_{\Gamma_{v}}K_{v}, \\
&\overline{Y}_{v}:=Bl_{\Gamma_{v}}Y_{v}
\end{aligned}
\]

With this notation, the finite $2:1$ morphism 
$$\widehat{\varepsilon}_{v}:=\widehat{\varepsilon}_{1,v}=\widehat{\varepsilon}_{2,v}: \widehat{Y}_{v}\rightarrow \widehat{K}_{v}$$
is a double cover between smooth varieties and is branched over the smooth divisor $\widehat{\Sigma}_{v}$ (see (1) of Corollary \ref{global diagram cover}).
Hence, $\widehat{\varepsilon}_{v}$ realizes $\widehat{K}_{v}$ as the quotient of   $\widehat{Y}_{v}$ under 
the action of the associated involution $\widehat{\tau}_{v}:\widehat{Y}_{v}\rightarrow \widehat{Y}_{v}$.

This permits to reconstruct  $\widehat{K}_{v}$ starting from the IHS manifold $\underline{Y}_{v}$ 
of  $K3^{[3]}$--type and using only birational modifications  of smooth projective varieties and the finite $2:1$ morphism 
$\widehat{\varepsilon}_{v}$. 

The following commutative diagram contains all the  varieties and maps that we will use.

\[
\xymatrix{
 & 
\ar@(ur,r)[]^{\widehat{\tau}_{v}}\widehat{Y}_{v}  
\ar[d]_{\widehat{\varepsilon}_{v}}  
\ar[dr]^{\beta_{v}} & & \\
 &  \widehat{K}_{v}  \ar[dl]_{\rho_{v}}   & \ar@(ur,r)[]^{\overline{\tau}_{v}}\overline{Y}_{v} \ar[dr]^{h_{v}} & \\ 
\widetilde{K}_{v}& & & \underline{Y}_{v} \ar@(ur,r)@{-->}[]^{\underline{\tau}_{v}}\\
}
\]
Here $\beta_{v}$ is the blow up map of $\overline{Y}_{v}$ along the smooth subvariety $\overline{\Delta}_{v}$ and 
$\overline{\tau}_{v}$ is the involution associated with $\overline{\tau}_{v}$ (see Remark \ref{involuzioni}).

Notice that this diagram contains only maps between smooth varieties that appear in (2) of Corollary  \ref{global diagram cover}
and any diagonal map that appears is the blow up of a smooth variety  along a smooth subvariety.

The IHS manifold $\underline{Y}_{v}$ carries a rational  involution 
$\underline{\tau}_{v}$ whose indeterminacy locus $J_{v}$ is the disjoint union of $256$ projective $3$--dimensional spaces
(see Remark \ref{rmk:ratinv}). The rational involution $\underline{\tau}_{v}$ lifts to the regular involution 
$\overline{\tau}_{v}$ on the blow up $\overline{Y}_{v}$ of $\underline{Y}_{v}$ along $J_{v}$, which in turn
 lifts to the involution $\widehat{\tau}_{v}:\widehat{Y}_{v}\rightarrow \widehat{Y}_{v}$ 
on the blow up of $\overline{Y}_{v}$ along the fixed locus $\overline{\Delta}_{v}$  of $\overline{\tau}_{v}$
(see (2) of Corollary \ref{cor:rambra}). Finally, the quotient $\widehat{K}_{v}$ of 
$\widehat{Y}_{v}$ modulo  $\widehat{\tau}_{v}$ is the blow up of  $\widetilde{K}_{v}$ along 
the union $\widetilde{\Omega}_{v}$ of 256 disjoint copies of the smooth $3$--dimensional quadric $G$.

The strategy to compute the Hodge numbers of $\widetilde{K}_{v}$ is the following.
Since $\widehat{K}_{v}$ is
the quotient of   $\widehat{Y}_{v}$ by 
the action of $\widehat{\tau}_{v}$, the Hodge numbers of $\widehat{K}_{v}$ that determine the Hodge numbers of $\widetilde{K}_{v}$ are the  
$\widehat{\tau}_{v}$-invariant Hodge numbers of  $\widehat{Y}_{v}$. 
The Hodge numbers of $\widehat{Y}_{v}$ can be easily computed in terms of the known Hodge numbers 
of the IHS manifold $\underline{Y}_{v}$ of  $K3^{[3]}$ type, and the action of $\widehat{\tau}_{v}$ is determined by the action
of the rational involution $\underline{\tau}_{v}:\underline{Y}_{v}\dashrightarrow \underline{Y}_{v}$ (see Remark \ref{rmk:ratinv}) 
on the Hodge groups of $\underline{Y}_{v}$. Finally, by Markman's monodromy results, 
this  action only depends on its part on the second cohomology group that is easy to compute.

Following this strategy, it turns out that the Betti numbers of $\widetilde{K}_{v}$ can be computed without 
considering the action $\widehat{\tau}_{v}$ on the cohomology of $\widehat{K}_{v}$.

\begin{prop}
The odd Betti numbers of $\Kvt$ are zero and the even ones are 
\[
h^2(\Kvt)= 8, \quad h^4(\Kvt)=199, \quad h^6(\Kvt)=1504.
\]
\end{prop}

\begin{proof} 
By Proposition \ref{prop:birh}, $\underline{Y}_{v}$ is deformation equivalent to the Hilbert scheme parametrizing length $3$ subschemes on a $K3$
surface, hence its odd Betti numbers are zero. 
By construction $\overline{Y}_{v}$ is the blow up of $\underline{Y}_{v}$ along $256$ disjoint projective spaces. Since the odd cohomology of the projective space is trivial, the same holds for $\overline{Y}_{v}$. 
By definition, $\widehat{Y}_{v}$ is the blow up of $\overline{Y}_{v}$ along $\overline{\Delta}_{v}$. We have already recalled that the ramification locus  $\overline{\Delta}_{v}$ of $\overline{\varepsilon}_{v}$ is isomorphic to the corresponding branch locus $\overline{\Sigma}_{v}$ which, by Remark
\ref{rmk:esplsigma}, is isomorphic to $(Bl_{(A\times A^{\vee})[2]}(A\times A^{\vee}))\slash \pm 1$.
As the odd cohomology classes of a torus are always antiinvariant under the action of $\pm 1$, the odd Betti numbers  of $\widehat{Y}_{v}$ are zero.
Since $\rho_{v}\circ \widehat{\varepsilon}_{v}: \widehat{Y}_{v}\rightarrow \Kvt$ is a regular surjective map between smooth projective varieties,
the rational cohomology of $\Kvt$ injects into the rational cohomology of $\widehat{Y}_{v}$. Hence the odd Betti numbers of $\Kvt$
are zero.
 
We already know that $h^2(\Kvt)= 8$ \cite{OG2} and that $\chi_{top}(\Kvt)=1920$ \cite{rap_phd}. 
The result follows using Salamon's formula \cite{Salamone}, which gives linear relations among the Betti numbers of a $2n$--dimensional irreducible holomorphic symplectic variety
\[
2\sum_{j=0}^{2n}(-1)^j(3j^2-n)b_{2n-j}=nb_{2n}.
\]
In our case this yields
\[
18 b_4+90 b_2+210=3b_6,
\]
Solving these two equations, we obtain the proposition.
\end{proof}

In order to determine the Hodge numbers of $\Kvt$, we first relate 
the Hodge numbers of $\Kvt$ with the $\overline{\tau}_{v}$--invariant Hodge numbers of $\overline{Y}_{v}$.
More specifically, we have the following lemmas.

\begin{lem} \label{Kt and Ktt} 
\begin{enumerate} 
\item{The following equalities of Hodge numbers hold 
\[
\begin{array} {l}
h^{p,q}(\widehat{K}_{v})= h^{p,q}(\Kvt) \;\; {\rm if} \;\;p\ne q,\\
h^{1,1}(\widehat{K}_{v})= h^{1,1}(\Kvt) +{256}, \\
h^{2,2}(\widehat{K}_{v})= h^{2,2}(\Kvt) +{ 512}, \\
h^{3,3}(\widehat{K}_{v})= h^{3,3}(\Kvt) +{ 512}. \\
\end{array}
\]}
\item{The vector space $H^{p,q}(\widehat{Y}_{v})^{\widehat{\tau}_{v}}$ of $\widehat{\tau}_{v}$ invariant $(p,q)$-forms on $\widehat{Y}_{v}$ 
is isomorphic to $H^{p,q}(\widehat{K}_{v})$.}
\end{enumerate}

\end{lem}
\begin{proof}
(1) follows from the fact that $r: \widehat{K}_{v}  \to \Kvt$ is the blow up along the $256$ quadrics $G_i \subset \Kvt$, and that the cohomology of a $3$--dimensional quadric is one--dimensional in even degrees and zero otherwise.
(2) holds because $\Kvt\simeq  \widehat{Y}_{v}\slash \widehat{\tau}_{v}$ (see Remark \ref{involuzioni}).
\end{proof}

In the following lemma we set ${a\choose { b}}:=0$ if $b>a$ or $b<0$.

\begin{lem} \label{X and Y} The  $\widehat{\tau}_{v}$--invariant Hodge numbers of $\widehat{Y}_{v}$
and the  $\overline{\tau}_{v}$--invariant Hodge numbers of $\overline{Y}_{v}$ are related in the following way
\[
\begin{array} {lll}
h^{p,q}(\widehat{Y}_{v})^{\widehat{\tau}_{v}}= h^{p,q}(\overline{Y}_{v})^{\overline{\tau}_{v}}=0, & &\text{ for }  \, p+ q \, \text{ odd }   , \\
h^{p,q}(\widehat{Y}_{v})^{\widehat{\tau}_{v}}= h^{p,q}(\overline{Y}_{v})^{\overline{\tau}_{v}}+{4 \choose {p-1}}{4\choose {q-1}}, 
 & &\text{ for }  \, p+ q \, \text{ even } \, \text{ and } \,\, p \neq q  ,\\
h^{p,p}(\widehat{Y}_{v})^{\widehat{\tau}_{v}}= h^{p,p}(\overline{Y}_{v})^{\overline{\tau}_{v}}+{4 \choose {p-1}}^{2}, & & \text{ for } \, p=0,1,5,6\, ,\\
h^{p,p}(\widehat{Y}_{v})^{\widehat{\tau}_{v}}= h^{p,p}(\overline{Y}_{v})^{\overline{\tau}_{v}}+{4 \choose {p-1}}^{2}+256, & & \text{ for } \, p=2,3,4 \, .
\end{array}
\]
\end{lem}
\begin{proof}
The morphism $\beta_{v}: \widehat{Y}_{v}\to \overline{Y}_{v}$ is the blow up of $\overline{Y}_{v}$ along a smooth subvariety isomorphic to $\overline{\Delta}_{v}$. 
By Corollary \ref{cor:rambra} and Remark \ref{rmk:esplsigma}, the variety  $\overline{\Delta}_{v}$ is isomorphic to
$Bl_{A \times A^{\vee}[2]} (A \times A^{\vee}  \slash \pm 1)$, hence its  Hodge numbers are the $\pm 1$--invariant Hodge numbers of  $Bl_{A \times A^{\vee}[2]}$. In other words 
\[
\begin{array} {ll}
h^{p,q}(\overline{\Delta}_{v})=0,  &\text{ for }  \, p+ q \, \text{ odd }  , \\
h^{p,q}(\overline{\Delta}_{v})=h^{p,q}(A \times A^{\vee}),  & \text{ for }  \, p+ q \, \text{ even}  \,\, \text{ and } \, p \neq q ,\\
h^{p,p}(\overline{\Delta}_{v})=h^{p,p}(A \times A^{\vee}), & \text{ for } \, p=0,4\, ,\\
h^{p,p}(\overline{\Delta}_{v})=h^{p,p}(A \times A^{\vee})+256, & \text{ for } \, p=1,2,3\, .
\end{array}
\]
With $h^{p,q}(A \times A^{\vee})={4 \choose p}{4\choose q}$. 
As a consequence (see Theorem 7.31 \cite{voi}), the Hodge numbers of $\widehat{Y}_{v}$ 
satisfy   
\[\begin{array} {lll}
h^{p,q}(\widehat{Y}_{v})= h^{p,q}(\overline{Y}_{v})=0, & &\text{ for }  \, p+ q \, \text{ odd }   , \\
h^{p,q}(\widehat{Y}_{v})= h^{p,q}(\overline{Y}_{v})+{4 \choose {p-1}}{4\choose {q-1}}, 
 & &\text{ for }  \, p+ q \, \text{ even } \, and \,\, p \neq q  ,\\
h^{p,p}(\widehat{Y}_{v})= h^{p,p}(\overline{Y}_{v})+{4 \choose {p-1}}^{2}, & & \text{ for } \, p=0,1,5,6\, ,\\
h^{p,p}(\widehat{Y}_{v})= h^{p,p}(\overline{Y}_{v})+{4 \choose {p-1}}^{2}+256, & & \text{ for } \, p=2,3,4 \, .
\end{array}
\]
The lemma follows, since the classes in $h^{p,q}(\widehat{Y}_{v})$ that come from $\overline{\Delta}_{v}$ are the pushforward of cohomology classes of
the exceptional divisor $\widehat{\Delta}_{v}$ which, by Corollary \ref{cor:rambra}, is the fixed locus of
$\widehat{\tau}_{v}$.
\end{proof}

It remains to determine the $\overline{\tau}_{v}$--invariant Hodge numbers $h^{p,q}(\overline{Y}_{v})^{\overline{\tau}_{v}}$ of $\overline{Y}_{v}$.
This will be done by relating the action in cohomology of $\overline{\tau}_{v}$ with the monodromy operator 
\[
m(\underline{\tau}_{v}): H^{\bullet}(\underline{Y}_{v})\rightarrow H^{\bullet}(\underline{Y}_{v})
\]
associated to the birational involution  $\overline{\tau}_{v}$.

To explain this relation first let us recall some details on the definition of $m(\underline{\tau}_{v})$.
\begin{rem}\label{monod}
By Theorem 2.5 of \cite{huykc} there exist   smooth proper families of IHS manifolds
$\underline{\mathcal{Y}}^{'}_{v}\rightarrow S$ and $\underline{\mathcal{Y}}_{v}\rightarrow S$ over a $1$--dimensional disk $S$ such that both 
the central fibers are isomorphic to $\underline{Y}_{v}$ and there exists a rational $S$-morphism 
$\underline{\mathcal{T}}_{v}: \underline{\mathcal{Y}}^{'}_{v}\dashrightarrow \underline{\mathcal{Y}}_{v}$ 
sending $\underline{\mathcal{Y}}^{'}_{v}\setminus J_{v}$ isomorphically to $\underline{\mathcal{Y}}_{v}\setminus J_{v}$
and restricting to $\underline{\tau}_{v}$ on central fibers.

By specializing the closure in  $ \underline{\mathcal{Y}}^{'}_{v}\times_{S} \underline{\mathcal{Y}}_{v}$ of the graph of 
$\underline{\mathcal{T}}_{v}$ over the central fiber, we obtain a pure $6$ dimensional cycle $\Upsilon$ on $\underline{Y}_{v}\times \underline{Y}_{v}$.

By definition $m(\underline{\tau}_{v})$ is the Hodge ring automorphism of $H^{\bullet}(\underline{Y}_{v})$ obtained as the associated correspondence of the cycle
$\Upsilon$. In our case, since $\underline{\mathcal{T}}_{v}$ induces an isomorphism between  $\underline{\mathcal{Y}}^{'}_{v}\setminus J_{v}$ and  
$\underline{\mathcal{Y}}_{v}\setminus J_{v}$ and restricts to $\underline{\tau}_{v}$ on central fibers, it follows that
\[\Upsilon=\Upsilon_{\underline{\tau}}+\sum_{i}m_{i}J_{i,v}\times J_{i,v}\]
where $\Upsilon_{\underline{\tau}}$ is the closure of the graph of $\underline{\tau}$ and the $m_{i}$'s are non negative integers 
\footnote{Using Key formula of Proposition 6.7 \cite{ful},
it can be shown  that $m_{i}=1$ for every $i$}.

\end{rem}

\begin{lem}\label{3fin}
\begin{enumerate}
\item{For every $i$ the cohomology class $[J_{i,v}]\in H^{6}(\underline{Y}_{v})$ of $J_{i,v}$ is $m(\underline{\tau}_{v})$--antiinvariant}
\item{The following relations between $m(\underline{\tau}_{v})$--invariant Hodge numbers of $\underline{Y}_{v}$ and 
$\overline{\tau}_{v}$--invariant Hodge numbers of $\overline{Y}_{v}$  hold: 
\[
\begin{array} {lll}
h^{p,q}(\overline{Y}_{v})^{\overline{\tau}_{v}}=h^{p,q}(\underline{Y}_{v})^{m(\underline{\tau}_{v})}, & &\text{ for }  \, p+ q \le 6  \text{ and } p\ne q   , \\
h^{1,1}(\overline{Y}_{v})^{\overline{\tau}_{v}}=h^{1,1}(\underline{Y}_{v})^{m(\underline{\tau}_{v})} +256,\\
h^{2,2}(\overline{Y}_{v})^{\overline{\tau}_{v}}=h^{2,2}(\underline{Y}_{v})^{m(\underline{\tau}_{v})} +256,\\
h^{3,3}(\overline{Y}_{v})^{\overline{\tau}_{v}}=h^{3,3}(\underline{Y}_{v})^{m(\underline{\tau}_{v})} +512.\\
\end{array}
\]
}
 
\end{enumerate}
\end{lem}
\begin{proof}
(1) As the differential of  the map $(h_{v},h_{v}\circ \overline{\tau}_{v}):\overline{Y}_{v}\rightarrow \underline{Y}_{v}\times \underline{Y}_{v}$ 
is everywhere injective, it induces an isomorphism $\Upsilon_{\underline{\tau}}\simeq \overline{Y}_{v}=Bl_{J_{v}} \underline{Y}_{v}$.
By Key formula of Proposition 6.7 \cite{ful}, the class $[J_{i,v}]$ is an eigenvector for correspondence $[\Upsilon_{\underline{\tau}}]^{*}$ 
induced by $[\Upsilon_{\underline{\tau}}]$ on  $H^{6}(\underline{Y}_{v})$ and, moreover, the corresponding eigenvalue $\lambda$ only depends on the normal bundle of  
$J_{i,v}$ in $\underline{Y}_{v}$, therefore it does not depend on $i$. 

On the other hand the correspondence induced by $J_{i,v}\times J_{i,v}$ on 
$H^{6}(\underline{Y}_{v})$ multiplies  $[J_{i,v}]$ by the degree of the third Chern class of its normal bundle in  $\underline{Y}_{v}$.
As this normal bundle is isomorphic to the cotangent bundle of $\mathbb{P}^{3}$, we obtain  $[J_{i,v}\times J_{i,v}]^{*} [J_{i,v}]=-4 [J_{i,v}]$.

It follows that $$m(\underline{\tau}_{v})[J_{i,v}]=(\lambda-4m_{i}) [J_{i,v}].$$
As $m(\underline{\tau}_{v})$ is an isomorphism on the integral cohomology, $\lambda-4m_{i}=\pm 1$ and the sign cannot depend on $i$.

It remains to exclude that $m(\underline{\tau}_{v})[J_{i,v}]=[J_{i,v}]$ for every $i$.
In this case, letting $A$ be the class of an ample divisor $A$ of $\underline{\tau}$, we have 
$$\int_{\underline{Y}_{v}}m(\underline{\tau}_{v})[A]^{3}\wedge [J_{i,v}]=\int_{\underline{Y}_{v}}[A]^{3}\wedge [J_{i,v}]>0.$$ 
Therefore, the line bundle associated with 
$m(\underline{\tau}_{v})[A]$ would be positive on the $J_{i,v}$'s and, as $A$ is ample, it would have positive degree on any curve on 
$\underline{Y}_{v}$. Finally, by Proposition 3.2 of \cite{huykc},   
$\underline{\tau}_{v}^{*} (A)$ would be an ample divisor and this is absurd because  $\underline{\tau}_{v}$ does not extend to an isomorphism.  

(2) 
Since $\underline{Y}_{v}$ is an IHS manifold of $K3^{[n]}$ type, its odd Betti numbers are trivial and the same holds for  $\overline{Y}_{v}$
as it is isomorphic to the blow up of $\underline{Y}_{v}$ along $J$ that is the disjoint union of $256$  projective spaces. Hence, we only need to 
consider the case where $p+q$ is even. 

If $p+q=6$, the exact sequences of the pairs  $(\overline{Y}_{v}, \overline{Y}_{v}\setminus \overline{\Gamma}_{v})$  and 
$(\underline{Y}_{v}, \underline{Y}_{v}\setminus J_{v})$, using excision and Thom isomorphism, give rise to  the commutative diagram 
\[
\xymatrix{
 0\ar[r]  & H^{0}(J_{v})\ar[d] \ar[r]^{r_1} &
 H^{6}(\underline{Y}_{v})\ar[d]^{h_{v}^{*}}\ar[r]^{r_2} &
H^{6}( \underline{Y}_{v}\setminus J_{v})\ar[d]\ar[r]& 0.\\
0\ar[r]  & H^{4}(\overline{\Gamma} _{v}) \ar[r]^{s_1} &
 H^{6}(\overline{Y}_{v})\ar[r]^{s_2}&
H^{6}( \overline{Y}_{v}\setminus \overline{\Gamma} _{v})\ar[r]& 0.
}
\]
In this diagram, $r_2$ and $s_{2}$ are surjective because the odd Betti numbers of $J_{v}$ and $\overline{\Gamma}_{v}$ are zero and  
$r_{1}$ is injective because the classes  $[J_{i,v}]$ are independent. This also implies that 
$H^{5}(\overline{Y}_{v}\setminus \overline{\Gamma}_{v})=H^{5}(\underline{Y}_{v}\setminus J_{v})=H^{5}(\underline{Y}_{v})=0$ and therefore $s_{1}$ is injective too.

As the intersection form of the middle cohomology of $\underline{Y}_{v}$ is nondegenerate, on  $H^{0}(J_{v})$ there is a splitting of Hodge structures
$$H^{6}(\underline{Y}_{v})=  H^{0}(J_{v})^{\perp} \oplus H^{0}(J_{v}),$$
where $H^{0}(J_{v})^{\perp}$ is the perpendicular to $H^{0}(J_{v})$ in ${H^{6}(\underline{Y}_{v})}$.
Since $m(\underline{\tau}_{v})$ acts as $-1$ on $H^{0}(J_{v})$ and the correspondence  $[J_{i,v}\times J_{i,v}]^{*}$ acts trivially on
$H^{0}(J_{v})^{\perp}$, we deduce that
$$H^{p,q}(\underline{Y}_{v})^{m(\underline{\tau}_{v})}=  ((H^{0}(J_{v})^{\perp})^{p,q})^{[\Upsilon_{\underline{\tau}}]^{*}},$$
for $p+q=6$.  

Since $h_{v}^{*} (H^{0}(J_{v})^{\perp})$ is included in the perpendicular
$H^{4}(\overline{\Gamma} _{v})^{\perp}$ to  $H^{4}(\overline{\Gamma} _{v})$ in $H^{6}(\overline{Y}_{v})$, the injective pull back  
$h_{v}^{*}$ induces an isomorphism of Hodge structures $H^{4}(\overline{\Gamma} _{v})^{\perp}\simeq H^{0}(J_{v})^{\perp}$.
It follows that the intersection form on the middle cohomology of $\overline{Y}_{v}$ is non degenerate on 
$H^{4}(\overline{\Gamma} _{v})$ and there is a splitting of Hodge structures
$$H^{6}(\overline{Y}_{v})=  H^{0}(\overline{\Gamma}_{v})^{\perp}\oplus H^{4}(\overline{\Gamma}_{v}).$$
Since $\overline{\tau}_{v}(\overline{\Gamma}_{v})= \overline{\Gamma}_{v}$
we also deduce
$$H^{p,q}(\overline{Y}_{v})=((H^{4}(\overline{\Gamma}_{v})^{\perp})^{p,q})^{\overline{\tau}_{v}} \oplus H^{p-1,q-1}(\overline{\Gamma}_{v})^{\overline{\tau}_{v}},$$
for $p+q=6.$

Moreover, the Hodge isomorphism $H^{4}(\overline{\Gamma} _{v})^{\perp}\simeq H^{0}(J_{v})^{\perp}$ identifies the action of $\overline{\tau}_{v}$
on $H^{4}(\overline{\Gamma} _{v})^{\perp}$ with the action of $[\Upsilon_{\underline{\tau}}]^{*}$ on $H^{0}(J_{v})^{\perp}$.

In fact, for any $\alpha\in H^{0}(J_{v})^{\perp}$, we have $[\Upsilon_{\underline{\tau}}]^{*}(\alpha)=(h_{v*}\circ\overline{\tau}_{v}^{*} \circ h_{v}^{*})(\alpha)$.
As $\overline{\tau}_{v}^{*}( h_{v}^{*}(\alpha))\in H^{4}(\overline{\Gamma}_{v})^{\perp}$ and since the kernel of $h_{v*}$ intersects trivially 
$H^{4}(\overline{\Gamma}_{v})^{\perp}$, the class $\overline{\tau}_{v}^{*}( h_{v}^{*}(\alpha))$ is the unique class in
$H^{4}(\overline{\Gamma}_{v})^{\perp}$ whose pushforward in $H^{6}(\underline{Y}_{v})$ is $[\Upsilon_{\underline{\tau}}]^{*}(\alpha)$.
Therefore $$\overline{\tau}_{v}^{*}( h_{v}^{*}(\alpha))= h_{v}^{*}([\Upsilon_{\underline{\tau}}]^{*}(\alpha)).$$

As a consequence,
$$H^{p,q}(\overline{Y}_{v})^{\overline{\tau}_{v}}=  H^{p,q}(\underline{Y}_{v})^{m(\underline{\tau}_{v})} \oplus H^{p-1,q-1}(\overline{\Gamma}_{v})^{\overline{\tau}_{v}}$$ and the result for $p+q=6$
follows because  $\overline{\Gamma}_{v}$ consists of $256$ copies of $I\subset \mathbb{P}(V)\times\mathbb{P}(V)$ on each of which $\overline{\tau}_{v}$
acts by exchanging the factors and the cohomology of each component of $\overline{\Gamma}_{v}$ comes by restriction from the cohomology of 
$\mathbb{P}(V)\times\mathbb{P}(V)$.

Finally, If $p+q=2k$, and $k=1$ or $k=2$,  as $J_{v}$ has codimension $3$ in $\underline{Y}_{v}$, restriction  
gives an isomorphism $H^{2k}(\underline{Y}_{v})\simeq H^{2k}(\underline{Y}_{v}\setminus J_{v})$ and there exists a Hodge decomposition 
$$H^{2k}(\overline{Y}_{v})=H^{2k}(\underline{Y}_{v})\oplus H^{2k-2}(\overline{\Gamma}_{v}).$$
Moreover, $H^{2k}(\overline{Y}_{v})$ can be seen as the subspaces of forms vanishing on $\overline{\Gamma}_{v}$ and 
it is stable under the action of $\overline{\tau}_{v}$. The same argument used in the case $p+q=6$ shows that the action of 
$\overline{\tau}_{v}$ on $H^{2k}(\underline{Y}_{v})$ coincides with the action of  $m(\underline{\tau}_{v})$, therefore   
$$ H^{p,q}(\overline{Y}_{v})^{\overline{\tau}_{v}}= 
 H^{p,q}(\underline{Y}_{v})^{m(\underline{\tau}_{v})} \oplus H^{p-1,q-1}(\overline{\Gamma}_{v})^{\overline{\tau}_{v}}.$$
As the invariant subspaces for action of $\overline{\tau}_{v}$ on the degree $0$ and the degree $2$ cohomology of each component of $\overline{\Gamma}_{v}$ 
has  dimension $1$, this proves  the lemma. 
\end{proof}

It remains to determine the $m(\underline{\tau}_{v})$--invariant Hodge numbers of $\underline{Y}_{v}$.
It will suffice to deal with the case where $A$ is  a general principally polarized abelian surface with
$NS(A)=\mathbb{Z}\Theta$ and where $v=(0,2\Theta,2)$. 

\begin{lem}\label{4fin} In this case the $m(\underline{\tau}_{v})$--invariant Betti numbers and Hodge numbers of $\underline{Y}_{(0,2\Theta,2)}$ are:
\[
\begin{array}{ccccc}
 (h^{0})^{m(\underline{\tau}_{v})}=1, & (h^2)^{m(\underline{\tau}_{v})}=7, & (h^{4})^{m(\underline{\tau}_{v})}=171,  & (h^{6})^{m(\underline{\tau}_{v})}=1178. \\ 
&&& \\
& &  (h^{2,0})^{m(\underline{\tau}_{v})}=1,  &    (h^{1,1})^{m(\underline{\tau}_{v})}=5,  \\
 & (h^{4,0})^{m(\underline{\tau}_{v})}=1, &  (h^{3,1})^{m(\underline{\tau}_{v})}=6,  & (h^{2,2})^{m(\underline{\tau}_{v})}= 157,  \\
 (h^{6,0})^{m(\underline{\tau}_{v})}=1, & (h^{5,1})^{m(\underline{\tau}_{v})}=5, & (h^{4,2})^{m(\underline{\tau}_{v})}=157, &  (h^{3,3})^{m(\underline{\tau}_{v})}=852. 
\end{array}
\]
\end{lem}
\begin{proof}
We first determine the weight $2$ $m(\underline{\tau}_{(0,2\Theta,2)})$--invariant Hodge numbers.
By Lemma \ref{Kt and Ktt}, Lemma \ref{X and Y}, and Lemma \ref{3fin} we have 
$$h^{2,0}(\underline{Y}_{v})^{m(\underline{\tau}_{v})}=h^{2,0}(\overline{Y}_{v})^{\overline{\tau}_{v}}=h^{2,0}(\widehat{K}_{v})=h^{2,0}(\widetilde{K}_{v})=1 $$
and
$$h^{1,1}(\underline{Y}_{v})^{m(\underline{\tau}_{v})}=h^{1,1}(\overline{Y}_{v})^{\overline{\tau}_{v}}-256=h^{1,1}(\widehat{K}_{v})-257=h^{1,1}(\widetilde{K}_{v})-1=5. $$

In order to compute the invariant part of the Hodge structure of $\underline{Y}_{(0,2\Theta,2)}$, we use a result of Markman \cite[Ex. 14]{mar}, which describes 
the action  of  monodromy operators 
on the Hilbert scheme of $3$ points on a K3 surface $S$ in terms of their action on the degree $2$ cohomology. 
Specifically, Markman proves that there are isomorphisms of representations of the monodromy group of    
$S^{[3]}$
\be \label{markmans decomposition}
\begin{aligned}
H^4(S^{[3]})&=\mathrm{Sym}^2 H^2(S^{[3]})\oplus H^2(S^{[3]}), \\
H^6(S^{[3]})&=\mathrm{Sym}^3 H^2(S^{[3]}) \oplus \Lambda^2 H^2(S^{[3]}) \oplus \C,
\end{aligned}
\ee
where $\C$ is a copy of the trival representation.

If $v=2(0,\Theta, 1)$, $\overline{Y}_{v}$ is birational to the
Hilbert scheme $S^{[3]}$ and hence there exists an isomorphism of Hodge rings
$k:H^{\bullet}(S^{[3]})\rightarrow H^{\bullet}(\overline{Y}_{v})$ and, moreover,
the Hodge involution $k^{-1} \circ m(\underline{\tau}_{v})\circ k$ is a
monodromy operator on $S^{[3]}$.
Moreover, the $m(\underline{\tau}_{v})$--invariant Hodge numbers of
$\overline{Y}_{v}$ coincide with the respective
$k^{-1} \circ m(\underline{\tau}_{v})\circ k$--invariant Hodge numbers of
$S^{[3]}$.
Since we know the weight-$2$ Hodge $m(\underline{\tau}_{v})$--invariant
numbers of  $\overline{Y}_{v}$,
we also know
the weight-$2$ Hodge $k^{-1}\circ m(\underline{\tau}_{v})\circ
k$--invariant numbers of  $\underline{Y}_{v}$ and,
using formulae
\eqref{markmans decomposition}, we can calculate all the $k^{-1}\circ
m(\underline{\tau}_{v})\circ k$--invariant numbers of $S^{[3]}$
and, therefore, all the  $ m(\underline{\tau}_{v})$--invariant numbers of
$\overline{Y}_{v}$.

To simplify the notation in the computation, set
$H^{p,q}_{+}:=H^{p,q}(\underline{Y}_{v})^{m(\underline{\tau}_{v})}$.
In particular, $H^{2,0}_{-}=H^{2,0}_{-}=0$.
By formulae \eqref{markmans decomposition} we obtain
  \[H^{4,0}_+= \mathrm{Sym}^2 H^{2,0}_{+}, \;\;  H^{3,1}_+=
(H^{2,0}_{+}\otimes H^{1,1}_{+}) \oplus  H^{2,0}_{+}, \;\;\]
\[ H^{2,2}_+= (H^{2,0}_{+}\otimes H^{0,2}_{+}) \oplus \mathrm{Sym}^2
H^{1,1}_{+} \oplus \mathrm{Sym}^2 H^{1,1}_{-} \oplus  H^{1,1}_{+},
\]
\[H^{6,0}_+=\mathrm{Sym}^3 H^{2,0}_{+}, \;\; H^{5,1}_+=\mathrm{Sym}^2
H^{2,0}_{+}\otimes H^{1,1}_{+},\,\,\]
\[ H^{4,2}_+= (\mathrm{Sym}^2 H^{2,0}_{+}\otimes H^{0,2}_{+})\oplus
(H^{2,0}_{+}\otimes  \mathrm{Sym}^2 H^{1,1}_{+})\oplus
(H^{2,0}_{+}\otimes  \mathrm{Sym}^2 H^{1,1}_{+})\oplus (H^{2,0}_{+}\otimes
H^{1,1}_{+}),\]
\[H^{3,3}_{+}=(H^{2,0}_{+}\otimes H^{1,1}_{+} \otimes H^{2,0}_{-})\oplus
\mathrm{Sym}^3H^{1,1}_{+}\oplus
(H^{1,1}_{+}\otimes \mathrm{Sym}^2 H^{1,1}_{-}) \oplus \]
\[(H^{2,0}_{+}\otimes H^{0,2}_{+})\oplus \Lambda^2 H^{1,1}_{+} \oplus
\Lambda^2 H^{1,1}_{-}\oplus \mathbb{C}, \]
which give  the invariant Hodge numbers.
Finally, the invariant Betti numbers are determined by the invariant Hodge
numbers.
\end{proof}

Now, a straightforward computation gives the Hodge numbers of O'Grady's
$6$--dimensional
IHS manifold.

\begin{thm}\label{thm:main} Let $\wt K$ be an IHS  manifold
of type OG6. The odd
Betti numbers of $\wt K$ are zero, its even Betti numbers are

\[
\begin{array}{ccccccl}
  b_0=1, & b_2=8, & b_4=199, &  b_6=1504, & b_8=199, & b_{10}=8, &
b_{12}=1, \\
\end{array}
\]
and its non--zero Hodge numbers are collected in the following table
\[
\begin{array}{ccccccl}
    &   &    &    H^{0,0}=1 &  & & \\
   &   &  H^{2,0}=1  &    H^{1,1}=6 & H^{0,2}=1 & & \\
     & H^{4,0}=1 &  H^{3,1}=12  &  H^{2,2}=173 & H^{1,3}=12 & H^{0,4}=1 & \\
  H^{6,0}=1 & H^{5,1}=6 & H^{4,2}=173 &  H^{3,3}=1144 & H^{2,4}=173 &
H^{1,5}=6 & H^{0,6}=1 \\
    & H^{6,2}=1 &  H^{5,3}=12  &  H^{4,4}=173 & H^{3,5}=12 & H^{2,6}=1 & \\
   &   &  H^{6,4}=1  &    H^{5,5}=6 & H^{4,6}=1 & & \\
   &   &    &   H^{6,6}=1 .&  & & \\
\end{array}
\]
\end{thm}
\begin{proof}
As Hodge and Betti numbers are stable under smooth K\"ahler deformations,
it will suffice to deal with
the case where $\wt K=\widetilde{K}_{(0,2\Theta,2)}$ and the  underlying
abelian surface
 $A$ is  a general abelian surface, whose Neron Severi group is
generated by the principal polarization
$\Theta$. In this case, Lemma \ref{Kt and Ktt}, Lemma \ref{X
and Y}, Lemma \ref{3fin}, and Lemma \ref{4fin} imply the result.
\end{proof}
Furthermore, the knowledge of the Hodge numbers is enough to compute the Chern numbers, as shown by Sawon \cite{giustino}. We have the following
\begin{cor}\label{prop:chern}
Let $\wt K$ be a manifold of OG6 type. Then $\int_{\wt K}c_2(\wt K)^3=30720$, $\int_{\wt K}c_2(\wt K)c_4(\wt K)=7680$ and $\int_{\wt K}c_6(\wt K)=\chi_{top}(\wt K)=1920$. \end{cor}
\begin{proof}
Let $\chi^p(\wt K)=\sum(-1)^qh^{p,q}(\wt K)$. In our case we have $\chi^0(\wt K)=4, \,\chi^1(\wt K)=-24$ and $\chi^2(\wt K)=348$. As shown in \cite[Appendix B]{giustino}, we have

$$
\int_{\wt K}c_2(\wt K)^3 = 7272\chi^0(\wt K)-184\chi^1(\wt K)-8\chi^2(\wt K),$$
$$ \int_{\wt K}c_2(\wt K)c_4(\wt K) = 1368\chi^0(\wt K)-208\chi^1(\wt K)-8\chi^2(\wt K),$$
$$\int_{\wt K}c_6(\wt K) = 36\chi^0(\wt K)-16\chi^1(\wt K)+4\chi^2(\wt K).$$

A direct computation yields our claim.
\end{proof}

\end{document}